\documentclass[12pt,reqno]{amsart}
\usepackage[a4paper,top=3cm,bottom=3cm,inner=3cm,outer=3cm]{geometry}

\usepackage{amsmath, amssymb, amsfonts, mathtools}
\usepackage{wasysym}

\usepackage{soul}

\usepackage{xcolor}
\definecolor{darkgreen}{rgb}{0,0.45,0}
\usepackage[
    colorlinks, 
    citecolor=blue, 
    urlcolor=darkgreen, 
    final, 
    hyperindex, 
    linkcolor = blue
]{hyperref}

\usepackage[capitalize]{cleveref}
\usepackage{xspace}
\usepackage{youngtab}
\usepackage{stackrel}

\usepackage{amsthm}
\usepackage{enumerate}

\theoremstyle{plain}
\newtheorem{thm}{Theorem}
\newtheorem*{thm*}{Theorem}

\newtheorem{lem}[thm]{Lemma}
\newtheorem*{lem*}{Lemma}
\newtheorem*{prop*}{Proposition}
\newtheorem{cor}[thm]{Corollary}

\theoremstyle{definition}
\newtheorem{defn}[thm]{Definition}
\newtheorem*{defn*}{Definition}
\newtheorem{example}[thm]{Example}

\theoremstyle{remark}
\newtheorem{rmk}[thm]{Remark}
\numberwithin{thm}{section}

\theoremstyle{plain} 
\newcommand{\thistheoremname}{}
\newtheorem*{genericthm*}{\thistheoremname}
\newenvironment{namedthm*}[1]
  {\renewcommand{\thistheoremname}{#1}%
   \begin{genericthm*}}
  {\end{genericthm*}}

\title[]{Categorical Lyapunov stability II: \\ Stability of Systems
}

\date{\today}
\author[]{Aaron D. Ames}
\author[]{S\'ebastien Mattenet}
\author[]{Joe Moeller}

\address{California Institute of Technology, Pasadena, CA 91125, USA}
\address{Université Catholique de Louvain, Pl. de l'Université 1, 1348 Ottignies-Louvain-la-Neuve, Belgium}

\email{ames@caltech.edu,
sebastien.mattenet@uclouvain.be,
jmoeller@caltech.edu}

\newcommand{\flow}{\phi}
\newcommand{\F}{\mathcal F}
\renewcommand{\L}{\mathcal L}
\newcommand{\I}{\textstyle \int}
\newcommand{\U}{U}
\newcommand{\D}{\mathrm D}

\newcommand{\ex}[1]{
\hspace*{\fill}
\textcolor{gray}{#1}
}

\newcommand{\exmath}[1]{
\tag*{
\textcolor{gray}{$#1$}
}}

\newcommand{\excenter}[1]{
\textcolor{gray}{#1}
}

\newcommand{\into}{\hookrightarrow}
\newcommand{\onto}{\twoheadrightarrow}
\newcommand{\id}{\mathrm{id}}
\newcommand{\inv}{^{-1}}

\newcommand{\norm}{\|\cdot\|_{x^*}}
\newcommand{\xnorm}[1]{\|#1\|_{x^*}}
\newcommand{\define}[1]{{\bf \boldmath{#1}}}

\newcommand{\maps}{\colon}
\newcommand{\To}{\Rightarrow}
\newcommand{\op}{^\mathrm{op}}
\newcommand{\N}{\mathbb N}
\newcommand{\R}{\mathbb R}
\newcommand{\Rplus}{{\mathbb R_{\geq0}}}

\newcommand{\T}{\mathrm {T}}
\newcommand{\K}{\mathcal K}

\newcommand{\category}[1]{\mathsf{#1}}
\newcommand{\B}{\category B}
\newcommand{\C}{\category C}
\renewcommand{\P}{\category P}

\newcommand{\namedcat}[1]{\mathsf{#1}}

\newcommand{\TComp}{T\mhyphen\Sys_\F}
\newcommand{\TFlow}{T\mhyphen\Flow}
\newcommand{\Dyn}{\namedcat{Dyn}}
\newcommand{\Man}{\namedcat{Man}}
\newcommand{\Meas}{\namedcat{Meas}}
\newcommand{\Set}{\namedcat{Set}}
\newcommand{\Sys}{\namedcat{Sys}}
\newcommand{\Flow}{\namedcat{Flow}}

\mathchardef\mhyphen="2D

\usepackage{tikz, tikz-cd}
\usetikzlibrary{positioning}

\begin{document}

\begin{abstract}
    Lyapunov's theorem provides a foundational characterization of stable equilibrium points in dynamical systems. In this paper, we develop a framework for stability for $\F$-coalgebras. We give two definitions for a categorical setting in which we can study the stability of a coalgebra for an endofunctor $\F$. One is minimal and better suited for concrete settings, while the other is more intricate and provides a richer theory. We prove a Lyapunov theorem for both notions of setting for stability, and a converse Lyapunov theorem for the second. 
\end{abstract}

\maketitle
\setcounter{tocdepth}{1}
\tableofcontents

\section{Introduction}
\label{sec:intro}

One of the most revealing features of a dynamical system is an equilibrium point, one which the dynamics does not ask to move. Identifying an equilibrium point is relatively simple as it is a property only to do with what the system specifies at the point. This contrasts with finer-grained classes of equilibrium points, which are often defined by their impact on their local neighborhood. For instance, every point that starts near a \emph{stable} equilibrium point will remain nearby as time progresses.

Verifying the stability of an equilibrium point directly is generally difficult, as it requires knowing the trajectory of every nearby point, to solve the defining differential equations of the system. Lyapunov devised an indirect method for certifying the stability of an equilibrium point without solving the system \cite{lyapunov1892general, lyapunov1992general}. 

First, consider a vector field $\dot x = f(x)$ on $E \subseteq \mathbb R^n$ which is ``strictly contracting'' towards an equilibrium point $x^*$, i.e., every vector $f(x)$ points roughly towards $x^*$ from $x$, rather than away. The function $V \maps E  \to \mathbb R$ given by $V(x_0) \coloneq \|x_0-x^*\|$ returns the distance of the initial condition $x_0$ to $x^*$. This function satisfies certain properties:  
\begin{itemize}
    \item It is ``positive definite'' with respect to $x^*$: $V(x_0) \geq 0$, and $V(x_0) = 0$ implies $x=x^*$.
    \item It is ``decrescent'' with respect to $f$: since the system is strictly contracting towards $x^*$ over time, we know that if $\flow_t(x_0)$ is the solution curve for the initial condition $x_0$, then $V(\flow_t(x_0)) \leq V(x_0)$, or said differently, $\frac{\partial V}{\partial x} \cdot f(x) \leq 0$ for all $x \in E$.  
\end{itemize}
These properties encapsulate that fact that the system is stable, which is intuitively obvious in this simple case. The value of $V$ is always non-negative, and any point will decrease its $V$-value over time, and we know this with certainty without actually solving the system. What Lyapunov realized is that for any system (not necessarily strictly contracting) \emph{any function} satisfying these two properties guarantees the stability of $x^*$. 

\begin{thm}[Classical Lyapunov theorem]
\label{thm:ClassicLyapunov}
    Let $\dot x = f(x)$ be an autonomous system of differential equations on $E \subseteq \mathbb R^n$, and $x^*$ an equilibrium point. Then $x^*$ is stable if and only if there is a function $V \maps E \to \mathbb R$ such that 
    \begin{enumerate}
        \item $V(x_0) \geq 0$, and $V(x_0) = 0$ implies $x=x^*$
        \item $\frac{\partial V}{\partial x} \cdot f(x) \leq 0$ for all $x \in E$.
    \end{enumerate}
\end{thm}

In \cite{CLT1}, we initiated a program to show that this classification of stable equilibria through ``Lyapunov functions'' is purely formal, and holds in much greater generality. Therein, we began with the assumption that we know the \emph{flow} of the system, that is, the trajectory of each point is encoded in an action $\flow \maps T \times E \to E$ of a monoid $T$ representing time ($\N$ for discrete-time, $\Rplus$ for continuous-time, etc.). We introduced a categorical axiomatic framework for studying stability of systems, which we termed \define{settings for stability}. This consists of three structured objects: a monoid object $T$ for time, a measurement object $R$ with a partial order and a distinguished element $0_R$, and an object of interest/state space $E$ which has a metric valued in $R$. The notion of system in this setting is an action of the monoid $T$ on the space of interest $E$. The main result of \cite{CLT1} is that Lyapunov's theorem holds in any setting for stability, i.e., that Lyapunov morphisms classify stable equilibria, and the converse holds with mild assumptions on $R$.

\begin{thm*}[\cite{CLT1}]
    Let $T \into E \onto R$ be a setting for stability, $\flow \maps T \times E \to E$ be a flow, and $x^* \maps 1 \to E$ an equilibrium point. If there exists a Lyapunov morphism $V$ for $\flow$, then $x^*$ is a stable equilibrium point. Moreover, if $R$ admits local suprema commuting with whiskering, then for any stable equilibrium point $x^*$ there exists a Lyapunov morphism $V$ for $\flow$.
\end{thm*}

The value of Lyapunov's method lies in the characterization of a global property (stability of equilibrium points) by a local property (the existence of a function that is decreasing in the direction in which the vector field is pointing). The purpose of this paper is to complete the story by giving Lyapunov conditions at the level of systems.

To capture the local nature of Lyapunov's method for vector fields, we complement the flow story using the theory of coalgebras of an endofunctor $\F \maps \C \to \C$, which we shall refer to as \define{$\F$-systems} \cite{rutten2000universal}. Smooth vector fields on manifolds are a special case of coalgebras of the endofunctor on the category $\Man$ of smooth manifolds given by sending a space to its tangent bundle (just the total space, not the associated projection map). That is, classical dynamical systems are $\T$-systems, where $\T \maps \Man \to \Man$ is the tangent bundle functor, which are additionally required to be sections of the bundle's projection. 

In order to connect local and global properties of systems, it is necessary to characterize when a flow is a ``solution'' to a system. This requires the development of additional theory for $\F$-systems when the endofunctor $\F$ is equipped with extra structure. Given this extra structure, we establish an equivalence between flows $\flow \maps T \times E \to E$ and systems ${f}_E \maps E \to \F E$, capturing the notion of unique solutions. 

In this paper, we introduce \define{settings for dynamic stability} where the time object is equipped with an $\F$-system $1_T \maps T \to \F T$ called the \define{unit clock} (generalizing $\dot{t} = 1$), the state space is equipped with a system of interest $f_E \maps E \to \F E$ (generalizing $\dot{x} = f(x)$), and the measurement object is equipped with a simple stable system $\sigma \maps R \to \F R$ (generalizing $\dot{y} = - \alpha(y)$), essentially for the point of comparison. 
\[
\underbrace{(T,1_T)}_{\textrm{unit clock}} \into \underbrace{(E,f_E)}_{\textrm{$\F$-system}} \onto \underbrace{(R,\sigma)}_{\textrm{stable system}}
\]

The key idea of the connection between the flow story and the system story of Lyapunov theory can be summarized by the following diagram. Lyapunov's theorem essentially tells us that if the right side lax commutes, then the left side lax commutes. The converse Lyapunov theorem tells us the left side lax commutes if the right side does.
\[
\begin{tikzcd}
    T \times E 
    \ar[r, "\flow"]
    \ar[d, "\pi"']&
    E
    \ar[r, "{f}_E"]
    \ar[d, "V"description]
    &
    \F E
    \ar[d, "\F V"]
    \\
    E
    \ar[r, "V"']
    \arrow[ur, Rightarrow]
    &
    R 
    \ar[r, "{0}_R"']
    \arrow[ur, Rightarrow]
    &
    \F R 
\end{tikzcd}
\]
The lax commuting square on the left is precisely the decrescent condition in the definition of Lyapunov morphisms on flows given in \cite{CLT1}. The lax commuting square on the right in turn is the decrescent condition we ask $V$ to satisfy to be a Lyapunov morphism on systems (cf. \cref{thm:lyapunovfsys}). Coupled with the results of \cite{CLT1}, this paper establishes necessary and sufficient conditions for the stability of $\F$-systems via Lyapunov morphisms, i.e., a complete characterization of the stability of systems. 

\begin{thm*}[Generalized Lyapunov Theorem, cf. Theorems \ref{thm:lyapunovfsys} and \ref{thm:Lyapconvsys}]
    Let ${f}_E \maps E \to \F E$ be a system in a setting for dynamic stability, and let $x^* \maps 1 \to E$ be an equilibrium point. Then $x^*$ is stable if there is a morphism $V \maps E \to R$ in $\C$ such that:
    \begin{enumerate}
        \item $V$ is positive definite with respect to $x^*$
        \item the following diagram lax commutes: 
        \[
        \begin{tikzcd}
            E
            \arrow[r, "V"]
            \arrow[d, "{f}_E"']
            &
            R
            \arrow[d, "{0}_R"]
            \arrow[dl, Rightarrow]
            \\
            \F E
            \arrow[r, "\F V"']
            &
            \F R
        \end{tikzcd}
        \exmath{\frac{\partial V}{\partial x}  {f}_E(x) \leq 0}\]
    \end{enumerate}
\end{thm*}

As in the case of Lyapunov morphisms on flows, we show that Lyapunov morphisms for systems are necessary for stability. This requires additional structural assumptions and corresponding machinery. When coupled with the sufficient conditions for stability, the end result is a complete characterization of the stability of systems via Lyapunov morphisms.

The axioms for a setting for dynamic stability provide a minimal list of conditions to check to guarantee that a Lyapunov theorem holds. This minimality also results in a lack of richness of the surrounding theory. Many natural notions in the classical Lyapunov theory are not expressible in the language afforded by those axioms. In \cref{sec:monoidal}, we provide a complementary set of axioms for a \emph{monoidal} setting for dynamic stability. Here, we ask for an operation which allows us to combine two system, equipping the functor $\F$ with a lax monoidal structure. Besides proving a Lyapunov theorem, these axioms allow us to define a derivative functor, an integration functor, and prove the following theorem relating them. Critical to the theory is the subcategory $\TComp \subseteq \Sys_\F$ of systems which are \emph{complete} relative to $T$ in the sense that each initial condition admits a unique solution in the form of a $T$-curve.

\begin{thm*}[Existence and Uniqueness, cf.\ \cref{thm:existenceuniqueness} and \cref{cor:monadic}]
    Assume that the unit clock $1_T \maps T \to \F T$ is itself $T$-complete. Given a $T$-complete system ${f}_E \maps E \to \F E$, there is a solution $\I {f}_E$ to ${f}_E$. This extends to a functor $\I \maps \TComp \to \TFlow$ with the action on morphisms defined trivially. There is a functor $\D \maps \TFlow \to \Sys_\F$ such that $\flow$ is a solution flow to $\D \flow$. If $\D\flow$ is $T$-complete for all $T$-flows $\flow$, then for any solution flow $\flow$ of a system ${f}_E \maps E \to \F E$, we have $\flow = \I{f}_E$. Moreover, $\I$ and $\D$ form an isomorphism of categories, and thus
    \[
    \flow = \I\D \flow, \qquad {f}_E = \D \I {f}_E. 
    \]
\end{thm*}

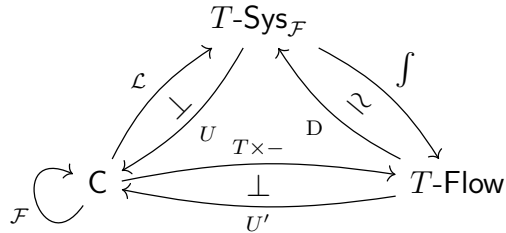
\[\begin{tikzcd}
    &
    \TComp
    \arrow[ddl, bend left = 15, "\U"]
    \arrow[ddr, bend left = 15, "\I"]
    \\\\
    \C
    \arrow[uur, phantom, sloped, "\bot"]
    \arrow[uur, bend left = 15, "\L"]
    \arrow[rr, bend left = 10, "T \times -"]
    \arrow[rr, phantom, "\bot"]
    &&
    \TFlow
    \arrow[uul, bend left = 15, "\D"]
    \arrow[uul, phantom, sloped, "\simeq"]
    \arrow[ll, bend left = 10, "\U'"]
    \arrow[from=3-1, to=3-1, loop, in=155, out=235, distance=10mm, "\F"]
\end{tikzcd}\]

\subsection{Overview of the paper}

In \cref{sec:flows}, we quickly review the definitions and results of \cite{CLT1}, which develops the categorical framework for Lyapunov theory where systems are represented as monoid actions or flows. In particular, the definition of \emph{setting for stability} and \emph{Lyapunov morphism} therein are given, and a Lyapunov theorem for such settings is stated. 
In \cref{sec:coalg}, we give the definition of $\F$-system (i.e., $\F$-coalgebra) for an endofunctor $\F$ on a category $\C$, and give several relevant examples. 
In \cref{sec:trajectory}, we introduce the definition of \emph{setting for dynamic stability}, which is a modification of setting for stability accounting for the data of an $\F$-system rather than a flow. \cref{thm:LyapTrajectory} establishes a coalgebraic characterization of Lyapunov morphisms in concrete settings, thus giving a Lyapunov theorem for $\F$-systems. 
In \cref{sec:monoidal}, we seek to remove the dependency on categorical concreteness. We introduce \emph{monoidal settings for dynamic stability}, a further modification of the notion of setting for stability which replaces references to points in an object with general maps into the object, i.e. ``generalized elements''. The proof of the Lyapunov theorem for these settings is more elegant, and does not require the underlying category to be concrete.
In \cref{sec:existunique}, we study the existence and uniqueness of solutions of systems in a setting for dynamic stability. \cref{thm:existenceuniqueness} gives conditions under which a setting admits an existence and uniqueness result, giving an equivalence between the category of $T$-flows and the subcategory of $\F$-systems which are complete relative to $T$.
In \cref{sec:converse}, we give conditions on a monoidal setting for stability under which the converse Lyapunov theorem holds. 

\section{Reviewing Systems as Flows}
\label{sec:flows}
The present work is building on that of \cite{CLT1}. For the convenience of the reader, we begin by reviewing some of the key definitions and results therein. We also recapitulate some of the key examples from \cite{CLT1} as they will be revisited in this paper in the context of systems. At the end of this section, we give an extension of the notion of stability from points to \emph{generalized elements}.

\begin{defn}
\label{def:setting}
    A \define{setting for stability}, $T \into E \onto R$, in $\C$ consists of:
    \begin{enumerate}
    \setcounter{enumi}{-1}
        \item[\textbf{S0:}] (setting) a category $\C$ with all finite products
        \item[\textbf{S1:}] (space) an object $E \in \C$
        \item[\textbf{S2:}] (time) an object $T$ equipped with the structure of a monoid: $(T, \oplus, 0_\oplus)$
        \item[\textbf{S3:}] (stable object) a posetal object $R \in \C$ equipped with a distinguished point $0_R \maps 1 \to R$ 
        \item[\textbf{S4:}] (distance) a morphism in $\C$, denoted by $
        d \maps E \times E \to R$, such that:
        \begin{itemize}
            \item $d \To 0$, 
            \item $\ker(d) \cong \Delta \maps E \to E \times E$, with $\Delta$ the diagonal map.
        \end{itemize} 
    \end{enumerate}
\end{defn}

The notation $T \into E \onto R$ is only intended to be impressionistic. There is no short exact sequence present. When looking at the arrow $T \into E$ the reader should think of a trajectory $\flow(\cdot, x_0)$ of a point $x_0 \in E$ under the influence of the flow $\flow$, which we define in \cref{def:flow} below. When looking at the arrow $E \onto R$ the reader should think of a ``norm'' $\norm$ relative to a point $x^*$, which we define in \cref{def:norm} below.

\begin{defn}
\label{def:flow}
    A \define{$T$-flow} $\flow \maps T \times E \to E$ on an object $E$ is an action of $T$ on $E$, meaning the following diagrams commute: 
    \[
    \begin{array}{c}
    \textit{Initialization:} \\
    \begin{tikzcd}
        E
        \arrow[r, "0_\oplus \times \id_E"]
        \arrow[dr, "\id_E"']
        &
        T \times E 
        \arrow[d, "\flow"]
        \\&
        E
    \end{tikzcd}
    \\
    \excenter{\flow(0, x) = x}
    \end{array}
    \qquad\qquad
    \begin{array}{c}
    \textit{Composition:} \\
    \begin{tikzcd}
        T \times T \times E
        \arrow[r, "\oplus \times \id_E"]
        \arrow[d, "\id_T \times \flow"']
        &
        T \times E
        \arrow[d, "\flow"]
        \\
        T \times E 
        \arrow[r, "\flow"']
        &
        X
    \end{tikzcd} 
    \\
    \excenter{\flow(t_1, \flow(t_2, x)) = \flow(t_1 + t_2, x)}
    \end{array}
    \]
    A \define{morphism of $T$-flows} $p \maps \flow_X \to \flow_Y$ is a map $p \maps X \to Y$ such that the following diagram commutes.
    \[
    \begin{tikzcd}
        T \times X
        \arrow[r, "\id_T \times p"]
        \arrow[d, "\flow_X"']
        &
        T \times Y
        \arrow[d, "\flow_Y"]
        \\
        X 
        \arrow[r, "p"']
        &
        Y
    \end{tikzcd}
    \exmath{p (\flow_X(t,x))= \flow_Y(t,p(x))}
    \]
    Let $\TFlow$ denote the category of $T$-flows and morphisms of $T$-flows.
\end{defn}

\begin{defn}
\label{def:norm}
    Given a point $x^* \maps 1 \to E$, the \define{norm} $\norm \maps E \to R$ relative to $x^*$ is the following composite \[\norm \coloneq \left[E \xrightarrow{\id \times x^*} E \times E \xrightarrow{d} R\right]\]
\end{defn}

\begin{defn}
\label{def:equlibriumpoint}
    An element $x^* \maps 1 \to E$ is an \define{equilibrium point of $\flow$} if the following diagram commutes:
    \[
    \begin{tikzcd}
        T
        \arrow[r, "id_T \times x^*"]
        \arrow[d, "!"']
        &
        T \times E
        \arrow[d, "\flow"]
        \\
        1
        \arrow[r, "x^*"']
        &
        E
    \end{tikzcd}
    \exmath{\flow_t(x^*) = x^*}
    \]
\end{defn}

\begin{defn}
\label{def:classKmorphism}
    A morphism $\alpha \maps R \to R$ is \define{class $\K$} if:
    \begin{itemize}
        \item $\alpha$ is an order-preserving map 
        \item $\alpha$ has an order-preserving inverse $\alpha\inv$
        \item $\alpha \circ 0_R = 0_R$.
    \end{itemize} 
\end{defn}

\begin{defn}
\label{def:stable}
    An equilibrium point $x^* \maps 1 \to E$ is \define{stable} if there is a class $\K$ morphism $\alpha$ such that the following diagram lax commutes:  
    \[
    \begin{tikzcd}
        T \times E 
        \arrow[rr, "\flow"]
        \arrow[d, "\pi"']
        &&
        E
        \arrow[d, "\|\cdot\|_{x^*}"]
        \\
        E
        \arrow[r, "\|\cdot\|_{x^*}"']
        \arrow[urr, Rightarrow, ""]
        &
        R
        \arrow[r, "\alpha"']
        &
        R
    \end{tikzcd}
    \exmath{\| \flow(t, x) \|_{x^*} \leq \alpha(\|x\|_{x^*})}
    \]
\end{defn}

\begin{defn}
\label{def:positivedef}
    A morphism $V \maps E \to R$ is \define{positive definite} with respect to $x^*$ if there exist class $\K$ morphisms $\underline\alpha, \overline \alpha \maps R \to R$ such that the following diagram lax commutes.
    \[\begin{tikzcd}
        &
        R
        \arrow[dr, "\overline \alpha"]
        \\
        E
        \arrow[ur, "\|\cdot\|_{x^*}"]
        \arrow[rr, "V"{name = V}, ""'{name = B}]
        \arrow[dr, "\|\cdot\|_{x^*}"']
        &&
        R
        \\&
        R
        \arrow[ur, "\underline \alpha"']
        \arrow[from = 1-2, to = V, Rightarrow, ""]
        \arrow[from = B, to = 3-2, Rightarrow, ""]
    \end{tikzcd}
    \exmath{\underline{\alpha}(\| x \|_{x^*}) \leq V(x) \leq \overline{\alpha}(\| x\|_{x^*})}
    \]
\end{defn}

\begin{defn}
\label{def:lyapunov}
   A morphism $V \maps E \to R$ is called a \define{Lyapunov morphism} for $\flow$ with equilibrium point $x^* \maps 1 \to E$ if:
    \begin{enumerate}
        \item $V$ is positive definite
        \item (decrescent) the following diagram lax commutes.
        \[
        \begin{tikzcd}
            T \times E
            \arrow[r, "\flow"]
            \arrow[d, "\pi"']
            &
            E
            \arrow[d, "V"]
            \\
            E
            \arrow[r, "V"']
            \arrow[ur, Rightarrow]
            &
            R
        \end{tikzcd}
        \exmath{ V( \flow(t, x) ) \leq V(x) }
        \]
    \end{enumerate}
\end{defn}

\begin{thm}[Lyapunov theorem for flows]
\label{thm:Lyap}
    Let $T \into E \onto R$ be a setting for stability, $\flow \maps T \times E \to E$ be a flow, and $x^* \maps 1 \to E$ an equilibrium point. If there exists a Lyapunov morphism $V$ for $\flow$, then $x^*$ is a stable equilibrium point. Moreover, if $R$ has local suprema commuting with whiskering, then for each stable equilibrium point $x^*$ there exists a Lyapunov morphism $V$ for $\flow$. 
\end{thm}

The precise definition of \emph{local suprema commuting with whiskering} can be found in \cite{CLT1}.

\begin{example}[Smooth Flows on Manifolds]
\label{ex:manifolds1}
Consider the category $\Man$ of smooth manifolds with boundary and smooth maps.  This yields a setting for stability: 
$$
\Rplus \into M \onto \Rplus
$$
in $\Man$ where: 
\begin{enumerate}
\setcounter{enumi}{-1}
    \item[\textbf{S0:}] (setting) $\Man$.
    \item[\textbf{S1:}] (space) a smooth manifold $M$ with empty boundary.
    \item[\textbf{S2:}] (time) $\Rplus$ with monoid structure given by addition $+$ and unit $0$.
    \item[\textbf{S3:}] (stable object) $\Rplus$ with point $0 \in \Rplus$ and posetal structure: $f \To g$, if $f(x) \geq g(x)$ for all $x \in M$. 
    \item[\textbf{S4:}] (distance) The metric $d \maps M \times M   \to   \Rplus$ obtained from the Riemannian metric $g_p \maps T_p M \times T_p M \to \R$ (which every smooth manifold admits) \cite{abraham2012manifolds}. 
\end{enumerate}
In this setting for stability, flows are smooth maps: $\flow \maps \Rplus \times M \to M$.  Per Definition \ref{def:equlibriumpoint}, an equilibrium point $x^*$ satisfies: $\flow_t(x^*) = x^*$.  The norm is $\xnorm{x} = d(x,x^*)$.  

Lyapunov's Theorem on flows states that $V \maps M \to \Rplus$ is a Lyapunov function if it is a smooth function satisfying: 
\begin{itemize}
    \item[(1)] $V$ is positive definite: there exists smooth class $\K$ functions $\underline{\alpha},\overline{\alpha}$ such that:  
\[
\underline{\alpha}(\| x \|_{x^*} ) \leq V(x) \leq 
\overline{\alpha}( \| x \|_{x^*} )
\]
\item[(2)] $V$ is decrescent along the flow: 
$V(\flow_t(x)) \leq V(x)$
\end{itemize}
These two conditions imply by Theorem \ref{thm:Lyap} that $x^*$ is a stable equilibrium point of the flow $\flow$, i.e., per Definition \ref{def:stable} if: 
\[
\| \flow_t(x) \|_{x^*} \leq \alpha(\| x \|_{x^*})
\]
for a class $\K$ function $\alpha$. Note that this agrees with the ``classical'' notion of stability:  for all $\epsilon > 0$, there exists a $\delta > 0$ such that: 
    \[
    \| x \|_{x^*} < \delta \qquad \mathrm{implies} \qquad 
    \| \flow_t(x) \|_{x^*} < \epsilon  \quad \forall t \in \Rplus
    \]
\end{example}

\begin{example}[Classical Stability]
\label{ex:classicstability1}
Consider an ODE $\dot{x} = f(x)$ with $x \in E$ and $E$ and open subset of $\mathbb{R}^n$. Under suitable assumptions, the solution to the ODE determines a flow:
\begin{eqnarray}
    \label{eqn:flow}
    \flow \maps \mathbb{R}_{\geq 0} \times E  \to E
\end{eqnarray}
satisfying $\dot{\flow}_t(x) = f(\flow_t(x))$, where $x \in E$ is the initial condition. 
To study the stability of this flow, we can consider a setting for stability in $\Set$: \[\Rplus \into E \onto \Rplus
\]
Even with this change of ambient category, we can use the same setting for stability as in \cref{ex:manifolds1} with the exception that $\textbf{S4}$ becomes: 
\begin{enumerate}
\setcounter{enumi}{-1}
     \item[\textbf{S4:}] (distance) the morphism:
    \begin{eqnarray}
        d \maps E \times E  & \to &  \Rplus \\
        (x,y) & \mapsto &  \| x - y \| \nonumber
    \end{eqnarray}
    where $\|\cdot\|$ is any (classical) norm in $\mathbb{R}^n$.
\end{enumerate}
We again obtain that an equilibrium point of the flow $\flow$ is stable if there exists a Lyapunov function that is positive definite, $\underline{\alpha}(\| x \|_{x^*} ) \leq V(x) \leq 
\overline{\alpha}( \| x \|_{x^*} )$, and decrescent: $V(\flow_t(x)) \leq V(x)$.  
Conversely, per Theorem \ref{thm:Lyap}, if $x^*$ is a stable equilibrium point of the flow $\flow$ then there exists a Lyapunov function. Namely: 
\[
V(x) \coloneq \sup_{t \in \Rplus} \xnorm{\flow_t(x)}.
\]
\end{example}

\begin{example}[Discrete-Time Stability]
\label{ex:discrete1}
    Consider a discrete-time dynamical system $x_{k+1} = F(x_k)$ with $x_k \in E \subset \mathbb{R}^n$ and $F \maps E \to E$. This results in a flow: 
    \begin{eqnarray}
    \label{eqn:discretetime}
        \flow \maps \N \times E & \to & E \\
        (k,x) & \mapsto & F^k(x) \nonumber
    \end{eqnarray}
    This leads to the system of stability: $\N \into E \onto \mathbb{R}_{\geq 0}$ defined by: 
    \begin{enumerate}
    \setcounter{enumi}{-1}
        \item[\textbf{S0:}] (setting) the category $\Set$, with terminal object $1 = \{*\}$. 
        \item[\textbf{S1:}] (space)  the space of interest is the set $E \subset \mathbb{R}^n \in \Set$.
        \item[\textbf{S2:}] (time) the object $\N$ is a monoid with $+$ addition and $0_\oplus = 0$. 
        \item[\textbf{S3:}] (stable object) the set $\mathbb{R}_{\geq 0}$ of the non-negative reals as in Examples \ref{ex:manifolds1} and \ref{ex:classicstability1}. 
        \item[\textbf{S4:}] (distance) the morphism $d \maps E \times E \to R$ is as in Example \ref{ex:classicstability1}. 
    \end{enumerate}
    An equilibrium point $x^*$ of the flow $\flow$ is a fixed point: $F(x^*) = x^*$. As a result, $x^*$ is stable if: $
    \| \flow_k(x) \|_{x^*} = \| F^k(x) \|_{x^*}  \leq \alpha(\| x \|_{x^*})
    $.  Per Theorem \ref{thm:Lyap}, stability of an equilibrium point $x^*$ follows from the existence of a Lyapunov function satisfying: 
    \begin{enumerate}
        \item $\underline{\alpha}(\| x \|_{x^*} ) \leq V(x) \leq \overline{\alpha}( \| x \|_{x^*} )$
        \item $V(\flow_k(x)) \leq V(x)$. 
    \end{enumerate}
    The second condition (for $k = 1$) gives the classic Lyapunov condition on a discrete-time dynamical system: 
    \[
    \nabla V(x) \coloneq V(F(x)) - V(x) \leq 0. 
    \]
    Therefore, the classic Lyapunov conditions for discrete-time dynamical systems are recovered. 
\end{example}

When working categorically, it is often necessary to avoid referring to the ``points'' of an object, as many categories have objects which are not composed of points, or at least not only composed of points. In a category with a terminal object, a map of the form $1 \to X$ can be thought of as a point in $X$. There are many categories where this fails to capture anything useful, e.g. the category of groups. The term ``generalized element'' is used to refer any map in a category $f \maps A \to E$. This notion is also a generalization of subobjects, but with even weaker conditions as we are not asking the map to be monic. 

\begin{defn}
    Let $\flow \maps T \times E \to E$ be a $T$-flow. A generalized element $x \maps A \to E$ is an \define{equilibrium} if the following diagram commutes.
    \[\begin{tikzcd}
        T \times A
        \arrow[r, "\id \times x"]
        \arrow[d, "\pi_A"']
        &
        T \times E
        \arrow[d, "\flow"]
        \\
        A
        \arrow[r, "x"']
        &
        E
    \end{tikzcd}\]
\end{defn}

In the classical setting, this corresponds to a subset consisting entirely of fixed points. This is not as interesting as a forward-invariant subset.

\begin{defn}
    Let $\flow \maps T \times E \to E$ be a $T$-flow. A generalized element $x \maps A \to E$ is \define{forward-invariant} if there exists a map $g \maps T \times A \to A$ such that the following diagram commutes.
    \[\begin{tikzcd}
        T \times A
        \arrow[r, "\id \times x"]
        \arrow[d, "g"']
        &
        T \times E
        \arrow[d, "\flow"]
        \\
        A
        \arrow[r, "x"']
        &
        E
    \end{tikzcd}\]
\end{defn}

Note that an equilibrium is forward-invariant with $g = \pi_A$. Also note that we are not explicitly asking for $g$ to be a $T$-flow on $E$. 

\begin{defn}
    Let $x \maps A \to E$ be a generalized element. The \define{norm to $x$}, $\|\cdot\|_x \maps E \to R$ is defined to be the infimum, i.e., the left Kan extension depicted below.
    \[
    \begin{tikzcd}
        E
        \arrow[rrd, "\|\cdot\|_x", bend left = 15, ""'{name = A}]
        \\
        E \times A
        \arrow[u, "\pi_E"]
        \arrow[r, "\id \times x"', ""{name = B}]
        &
        E \times E
        \arrow[r, "d"']
        &
        R
        \arrow[from = B, to = A, Rightarrow]
    \end{tikzcd}
    \]
    In the classical setting this is $\|y\|_x := \inf_{a \in A} \|y-x(a)\|$. The diagram above says that $\|y\|_x \leq \|y-x(a)\|$ for all $a \in A$, and being the infimum is just asking this to be the smallest such function. 
\end{defn}

With this definition of norm relative to a generalized element $x$, we may define stability of $x$, positive definiteness of a morphism $V \maps E \to R$ relative to $x$, and Lyapunov morphisms by using \cref{def:stable}, \cref{def:positivedef}, and \cref{def:lyapunov} respectively, with essentially no changes. Similarly, the proof of the (sufficient) Lyapunov theorem given in \cite{CLT1} gives the following generalization.

\begin{thm}
    Let $T \into E \onto R$ be a setting for stability, $\flow \maps T \times E \to E$ be a flow, and $x^* \maps A \to E$ be a generalized element. If there exists a Lyapunov morphism $V$, then $x^*$ is stable.
\end{thm}

\section{Systems as Coalgebras}
\label{sec:coalg}

In this paper, we take the view that systems can be represented by coalgebras.  This is a well-understood framework for systems---from labeled transition systems \cite{rutten2000universal} to dynamical, control and hybrid systems \cite{tabuada2009verification}.  

\begin{defn}
Let $\F \maps \C \to \C$ be an endofunctor. The category of $\F$-systems, denoted by $\Sys_\F$, is given by:
\begin{enumerate}
    \item (objects) \define{$\F$-systems}, consisting of a pair $(E, {f}_E)$ where $E \in \C$ is the object of \define{states} and ${f}_E \maps E \to F(E)$ in $\C$ is the \define{dynamics}. 
    \item (morphisms) A \define{morphism of $\F$-systems} $g \maps {f}_E \to {f}_{E'}$ is a morphism $g \maps E \to E'$ in $\C$ such that the following diagram commutes.
    \begin{equation}
    \label{eqn:fsystemmorphism}
    \begin{tikzcd}
        \F E
        \arrow[r, "\F g"]
        &
        \F E'
        \\
        E
        \arrow[u, "{f}_E"]
        \arrow[r, "g"']
        &
        E'
        \arrow[u, "{f}_{E'}"']
    \end{tikzcd}
    \end{equation}
\end{enumerate}
\end{defn}

\begin{rmk}
\label{rmk:functtocoalgisnat}
    Note that $\F$-systems are typically called \define{$\F$-coalgebras} \cite{rutten2000universal}. There is also the dual notion of an $\F$-algebra with the arrows reversed. We also note that category of $\F$-systems is a subcategory of the comma category $(\id_{\C} \downarrow \F)$ associated with the following diagram.
    \[\C \xrightarrow{\id_\C} \C \xleftarrow{\F} \C\]
    In fact, $\Sys_\F$ satisfies its own 2-dimensional universal property. The category $\Sys_\F$ admits a forgetful functor $U \maps \Sys_\F \to \C$ which sends a system $ f \maps E \to \F E$ to its underlying $\C$-object $E$, and there is a natural transformation $\eta \maps U \To \F \circ U \maps \Sys_\F \to \C$ the component of which corresponding to an object $ f \in \Sys_\F$ is the map $ f \maps E \to \F E$ itself.
    These data form the universal oplax cone of the \emph{oplax equalizer} of the parallel pair of functors $\id_\C$ and $\F$. 
    \[\begin{tikzcd}
        \Sys_\F
        \arrow[r, "U"]
        &
        \C
        \arrow[rr, "\id_\C", ""'{name = A}, bend left]
        \arrow[rr, "\F"', ""{name = B}, bend right]
        &{}&
        \C
        \arrow[from = A, to = B, Rightarrow, "\eta"]
    \end{tikzcd}\]
    This means that for any category $\mathsf D$, functor $Q \maps \mathsf D \to \C$ and natural transformation $\varepsilon \maps Q \To \F \circ Q$, there is a unique functor $\overline \varepsilon \maps \mathsf D \to \Sys_\F$ such that $U \circ \overline \varepsilon = Q$ and 
    \[\begin{tikzcd}
        &&
        \C
        \arrow[dr, "\id_\C"]
        \arrow[dd, Rightarrow, "\eta"]
        \\
        \mathsf D
        \arrow[urr, bend left = 15, "Q"]
        \arrow[r, "\overline \varepsilon"description]
        \arrow[drr, bend right = 15, "Q"']
        &
        \Sys_F
        \arrow[ur, "U"]
        \arrow[dr, "U"']
        &&
        \C
        \\&&
        \C
        \arrow[ur, "\F"']
    \end{tikzcd}
    \qquad=\qquad\begin{tikzcd}
        &
        \C
        \arrow[dr, "\id_\C"]
        \arrow[dd, Rightarrow, "\varepsilon"]
        \\
        \mathsf D
        \arrow[ur, "Q"]
        \arrow[dr, "Q"']
        &&\C
        \\&
        \C
        \arrow[ur, "\F"']
    \end{tikzcd}\]
    In summary, any functor $\overline \varepsilon$ into $\Sys_\F$ is uniquely determined by a functor into $\C$ paired with a natural transformation $\varepsilon \maps Q \To \F \circ Q$. 
\end{rmk}

\begin{example}[Vector Fields on Manifolds]
\label{ex:manifolds2}
    Let $\Man$ be the category of smooth manifolds with boundary and smooth maps. Then the \emph{tangent bundle functor} $\T \maps \Man \to \Man$ associates to a manifold its tangent bundle. Importantly, associated with this functor is the canonical projection $\pi \maps \T \To \id_\Man$ which has components $\pi_M \maps \T(M) \to M$. A \emph{smooth dynamical system} is a $\T$-system ${f} \maps M \to \T M$ such that $\pi \circ {f} = \id_M$. Dynamical systems as $\T$-coalgebras have been studied before, for example in \cite{Jubin} and \cite{Pavlovic_Fauser}. The category of smooth dynamical systems, $\Dyn$, is a full subcategory of $\Sys_\T$. As this category is of central importance, we explicitly note that the category $\Dyn$ has as: 
    \begin{itemize}
        \item (objects) smooth dynamical systems ${f} \maps M \to \T M$. 
        \item (morphisms)  $H \maps (M, {f}) \to (N, {g})$, where $H \maps M \to N$ is a smooth map between manifolds such that the following diagram commutes: 
        \begin{equation}
        \label{eqn:dynmorphism}
        \begin{tikzcd}
            \T M
            \arrow[r, "\T H"]
            &
            \T N
            \\
            M
            \arrow[u, " f"]
            \arrow[r, "H"']
            &
            N
            \arrow[u, " g"']
        \end{tikzcd}
        \end{equation}
        where $\T H(x, v) = (H(x), DH(x) v )$ is the \emph{tangent map}. Therefore, there is a morphism between dynamical systems if they are \emph{H-related}.

    \end{itemize}
\end{example}

\begin{example}[Discrete-Time Dynamical Systems]
\label{ex:discrete2}
    Let $\C$ be any category. A coalgebra of the identity functor $\C \to \C$ is just an object equipped with an endomorphism, ${f} \maps X \to X$, i.e., a discrete-time dynamical system in $\C$. In particular, $\C = \Set$ in \cref{ex:discrete1}. If $\C$ has finite products, countable coproducts, and products distribute over coproducts, then we can construction the $\C$-natural numbers as $\N_\C \coloneq \coprod_{n \in \N} 1$, which inherits a monoid structure from $\N$ in the obvious way. Notice that our assumption of the distributive property gives us \[\N_\C \times X \cong \left(\coprod_{n \in \N} 1\right) \times X \cong \coprod_{n \in \N} X.\] In this case, we obtain a flow $\flow \maps \N_\C \times X \to X$ from the universal property of the coproduct $\coprod_\N X$ as the copairing of countably many copies of $ f$. 
\end{example}

\begin{example}[Graphs]
\label{ex:graphs}
    Consider the covariant power set functor $\P \maps \Set \to \Set$. A $\P$-coalgebra consists of a set $X$ along with a map $ f \maps X \to \P(X)$. This amounts to the data of a binary relation $X$ where $x \in X$ gets mapped to the set of all points to which $x$ is related. We may interpret such a relation as a directed graph with $X$ as the set of vertices. Any ordered pair of vertices may have either one or no edges between them. 
\end{example}

\begin{example}[Labeled Transition Systems]
\label{ex:LTS1}
    Fix a set $A \in \Set$. An \define{$A$-labeled transition system} is a coalgebra of the composite functor $\P(A \times -) \maps \Set \to \Set$. Similarly to \cref{ex:graphs}, a labeled transition system (LTS) $ f \maps X \to \P(A \times X)$ can be interpreted as a sort of graph, where the set of vertices is $X$, and if $(a,y) \in  f(x) \subseteq A \times X$, then we say there is an edge $x \xrightarrow ay$ with label $a \in A$. 
\end{example}

\begin{example}[Markov kernels]
\label{ex:MarkovKernels1}
    Let $\Meas$ denote the category of measurable spaces. Let $D \maps \Meas \to \Meas$ denote the functor which sends a measurable space $(X, \Sigma)$ to the measurable space of probability distributions on $(X,\Sigma)$, known as the \define{Giry monad} \cite{Giry}. A $D$-system $ f \maps X \to D(X)$ is a Markov kernel.
\end{example}

\begin{example}
    Let $\C$ be a category with a terminal object $1$, and consider the constant functor $\Delta 1 \maps \C \to \C$. It holds essentially by definition that $\Sys_{\Delta1} \simeq \C$. 
\end{example}

\section{Lyapunov's Theorem on Trajectories}
\label{sec:trajectory}

A setting for stability as in \cref{def:setting} can be thought of as the ``base space'' on which stability can be defined. This is restricted by the fact that it only treats a system as its ``behavior'', the flow you get after solving a system of differential equations for instance. In order to capture the true utility of Lyapunov's theorem, this framework must be extended to include ``dynamics'', or the ``specification'' of a system. 

In this paper, we use $\F$-systems as our notion of system specification. In this section we give our first pass at settings for dynamic stability, and prove a relevant Lyapunov theorem. This prove depends on the concreteness of the underlying category of the setting, that is, the similarity of $\C$ to $\Set$. In particular, in a concrete category, a flow $\flow \maps T \times E \to E$ can be studied through its trajectories $\flow(-,x_0) \maps T \to E$. This is due to the fact that in a concrete category, an object is essentially composed of points. In future sections, we give a refinement of the following definition which is more conceptually robust and does not depend on concreteness. 

\begin{defn}
\label{def:settingdynamicstabilty}
    A \define{setting for dynamic stability}, $(T,1_T) \into (E,f_E) \onto (R,\sigma)$, in a category of $\F$-systems $\Sys_{\F}$ consists of: 
    \begin{enumerate}
    \setcounter{enumi}{-1}
        \item[\textbf{D0:}] (setting) a setting for stability $T \into E \onto R$ in $\C$
        \item[\textbf{D1:}] (unit clock) An $\F$-system on the time object ${1}_T \maps T \to \F T$.
        \item[\textbf{D2:}] (stable system) $R$ is equipped with an $\F$-system ${\sigma} \maps R \to \F R$
        \item[\textbf{D3:}] (posetal structure) $\F R$ is a posetal object
        \item[\textbf{D4:}] (comparison lemma) the diagram on the left lax commuting implies that the diagram on the right lax commutes: 
        \[
        \begin{tikzcd}
            \F T
            \arrow[r, "\F \gamma"]
            &
            \F R
            \\
            T
            \arrow[u, "{1}_T"]
            \arrow[r, "\gamma"']
            &
            R
            \arrow[u, "{\sigma}"']
            \arrow[ul, Rightarrow, ""]
        \end{tikzcd}
        \qquad \mathrm{implies} \qquad 
        \begin{tikzcd}
            T 
            \arrow[r,bend left=35pt,"\gamma_0", ""'{name = A}]
            \arrow[r,bend right=35pt,"\gamma"', ""{name = B}]
            &
            R
            \arrow[from = A, to = B, Rightarrow]
        \end{tikzcd}
        \]
        \[
        \excenter{ \quad \dot{\gamma}(t) \leq {\sigma}(\gamma(t)) \quad \qquad  \textrm{implies} \qquad  \gamma(t) \leq \gamma(0)}
        \]
         where $\gamma_0 := [T \xrightarrow{!} 1 \xrightarrow{0_T} T \xrightarrow{\gamma} R ]$ with $0_T$ the monoid unit for $T$.
    \end{enumerate}
\end{defn}

\begin{rmk}
    In \cref{lem:Tselfcomplete} below, we show that under certain conditions, the unit clock system in Axiom \textbf{D1} generates and is generated by the monoid structure on $T$ leveraged in the context of flows. 
\end{rmk}

\begin{defn}
\label{def:solntrajectory}
    A map $c \maps T \to E$ is called a \define{trajectory} for $ f_E$ if it is a map of systems $c \maps {1}_T \to {f}_E$, i.e., if the following diagram commutes.
    \[\begin{tikzcd}
        \F T
        \arrow[r, "\F c"]
        &
        \F E
        \\
        T
        \arrow[u, "{1}_T"]
        \arrow[r, "c"']
        &
        E
        \arrow[u, " f_E"']
    \end{tikzcd}\]
    A $T$-flow $\flow \maps T \times E \to E$ is a \define{trajectory flow} for the $\F$-system ${f}_E \maps E \to \F E$ if for each initial condition $x_0 \maps 1 \to E$ the composite:
    \[
    \flow(\cdot,x_0) := \left[T \xrightarrow{\id \times x_0} T \times E \xrightarrow \flow E \right]
    \]
    is a trajectory.
\end{defn}

An equilibrium point $x^* \maps 1 \to E$, cf.\ Definition \ref{def:equlibriumpoint}, can be equivalently stated for trajectories via the following commutative diagram: 
\begin{eqnarray}
\label{def:equilibrium2}
    \hspace{2cm} 
    & 
    \begin{tikzcd}
    T
    \arrow[d, "!"']
    \arrow[dr, "\flow(\cdot{,}x^*)"]
    \\
    1
    \arrow[r, "x^*"']
    &
    E
    \end{tikzcd}
    & 
    \hspace{2cm} \ex{\flow(t,x^*) = x^*}
\end{eqnarray}
The stability of equilibrium points and Lyapunov morphisms, per \cref{def:stable} and \cref{def:lyapunov} can be certified for trajectories $\flow$ of an $\F$-system $f_M \maps M \to \F M$ as follows. 

\begin{defn}
    Recall that an object $R \in \C$ is \define{posetal} if for each object $A \in \C$ the hom-set $\C(A, R)$ is equipped with a partial order such that for any map $f \maps A \to B$, the induced map $f^* \maps \C(B,R) \to \C(A,R)$ is order-preserving \cite{CLT1}. A posetal structure on an object $R$ is said to be \define{induced pointwise} if for $f,g \maps A \to R$ if $fx\leq gx$ for all $x \maps 1 \to A$, then $f \leq g$.
\end{defn}

\begin{thm}[Lyapunov's theorem on trajectories]
\label{thm:LyapTrajectory}
    Let $(T,{1}_T) \into (E, f_E) \onto (R, {\sigma})$ be a setting for dynamic stability in a category $\C$ with the posetal structure on $R$ pointwise induced, and $\flow \maps T \times E \to E$ a trajectory. If $V \maps E \to R$ is a morphism such that: 
    \begin{enumerate}[(i)]
        \item $V$ is positive definite: 
        \[\begin{tikzcd}
            &
            R
            \arrow[dr, "\overline \alpha"]
            \\
            E
            \arrow[ur, "\|\cdot\|_{x^*}"]
            \arrow[rr, "V"{name = V}, ""'{name = B}]
            \arrow[dr, "\|\cdot\|_{x^*}"']
            &&
            R
            \\&
            R
            \arrow[ur, "\underline \alpha"']
            \arrow[from = 1-2, to = V, Rightarrow, ""]
            \arrow[from = B, to = 3-2, Rightarrow, ""]
        \end{tikzcd}\]
        \item the following diagram lax commutes: 
        \begin{equation}
        \label{eq:SysDecrescent}
        \begin{tikzcd}
            \F E
            \arrow[r, "\F V"]
            &
            \F R
            \\
            E
            \arrow[u, " f_E"]
            \arrow[r, "V"']
            &
            R
            \arrow[u, "{\sigma}"']
            \arrow[ul, Rightarrow, ""]
        \end{tikzcd}
        \end{equation}
    \end{enumerate}
    then $V \maps E \to R$ is a Lyapunov morphism for $\flow$, and thus $x^*\maps 1 \to E$ is stable. 
\end{thm}
\begin{proof}
    Since we assume $V$ is positive definite, we need only show that it is decrescent with respect to $\flow$. Let $x_0 \maps 1 \to X$. Consider the following pasting diagram. 
    \[\begin{tikzcd}[column sep = small]
        \F T
        \arrow[rr, "\F(\flow(\cdot{,}x_0))", outer sep = 5]
        &&
        \F E
        \arrow[rr, "\F V"]
        &&
        \F R
        \\
        T
        \arrow[u, "{1}_T"]
        \arrow[rr, "\flow(\cdot{,}x_0)"']
        &&
        E
        \arrow[u, "f_E"]
        \arrow[rr, "V"']
        &&
        R
        \arrow[u, "{\sigma}"]
        \arrow[ull, Rightarrow]
    \end{tikzcd}\]
    Letting $\gamma = V \circ \flow(\cdot{,}x_0)$, we conclude that the following square commutes by Axiom \textbf{D4} in \cref{def:settingdynamicstabilty}.
    \[
    \begin{tikzcd}[column sep = large]
        T
        \arrow[r, "~0_T"]
        \arrow[d, "\flow(\cdot{,}x_0)"']
        \arrow[rr,bend left=45pt,"\flow(0_T{,}x_0)"]
        &
        T
        \arrow[r, "\flow(\cdot{,}x_0)"]
        &
        E
        \arrow[d, "V"]
        \arrow[dll, Rightarrow]
        \\
        E
        \arrow[rr, "V"']
        &&
        R
    \end{tikzcd}
    \]
    This is equivalent to the following diagram lax commuting since $\flow$ is a $T$-flow. 
       \[
    \begin{tikzcd}[column sep = large]
        &
        T\times E
        \arrow[dr,"\pi"]
        &
        \\
        T
        \arrow[rr, "\flow(0_T{,}x_0)"]
        \arrow[d, "\id\times x_0"']
        \arrow[ur,"\id\times x_0"]
        &
        &
        E
        \arrow[d, "V"]
        \arrow[dll, Rightarrow]
        \\
        T\times E
        \arrow[r,"\flow"']
        &
        E
        \arrow[r, "V"']
        &
        R
    \end{tikzcd}
    \]
    Since this lax commutes for each point $x_0 \maps 1 \to E$, then by the assumption that $R$ is pointwise induced, the following diagram lax commutes.
    \[
    \begin{tikzcd}
            T \times E
            \arrow[r, "\flow"]
            \arrow[d, "\pi"']
            &
            E
            \arrow[d, "V"]
            \\
            E
            \arrow[r, "V"']
            \arrow[ur, Rightarrow]
            &
            R
        \end{tikzcd}
    \]
    Thus $V$ is a Lyapunov morphism for solutions $\flow$ of $f_E$, and $x^* \maps 1 \to E$ is stable. 
\end{proof}

\begin{example}[Classical dynamical stability]
\label{ex:classicstability2}
    Recall the functor $\T \maps \Man \to \Man$ from \cref{ex:manifolds2}. The category $\Man$ is concrete.
    \begin{enumerate}
    \setcounter{enumi}{-1}
        \item[\textbf{D0:}] (setting) We use the setting for stability given in \cref{ex:classicstability1}.
        \item[\textbf{D1:}] (unit clock) The unit clock $1 \maps \Rplus \to \T(\Rplus)$ assigns the vector $1$ to each point.
        \item[\textbf{D2:}] (stable object) Let $\sigma = 0 \maps \Rplus \to \T\Rplus$ be the constant $0$ vector field. 
        \item[\textbf{D3:}] (posetal structure) Since $\T \Rplus \cong \Rplus \times \R$, we give $\T \Rplus$ the lexicographic ordering, meaning that $(t,v) \leq (t',v')$ if $t < t'$ or $t=t'$ and $v \leq v'$.
        \item[\textbf{D4:}] (comparison lemma)  If a curve $\gamma \maps \Rplus \to \Rplus$ makes the left diagram in Axiom \textbf{D4} lax commute, then $\frac{\partial \gamma}{\partial t} \leq 0$. This implies that $\gamma(t) \leq \gamma(0)$ which is what the right diagram of Axiom \textbf{D4} says. 
    \end{enumerate}
    The coalgebraic decrescent condition $\F V \circ  f_E \leq  \sigma \circ V$ translates to $\frac{\partial V}{\partial x} \cdot  f(x) \leq \sigma(V(x))$, which is the classical Lyapunov decrescent condition. 
\end{example}

\begin{example}[Discrete-time stability]
\label{ex:discrete3}
    Recall from \cref{ex:discrete2} that an $\id_\C$-system is a discrete-time system in $\C$. In particular, classical discrete-time systems are $\id_\Set$-systems.
    \begin{itemize}
        \item[\textbf{D0:}] (setting) The setting for stability is that of \cref{ex:discrete1}.
        \item[\textbf{D1:}] (unit clock) Since the time object is $\N$, we define the unit clock system $ 1 = +1 \maps \N \to \N$ by $n \mapsto n+1$.
        \item[\textbf{D2:}] (stable system) Since the measurement object is $R = \Rplus$, we give it the coalgebra $\id_\Rplus \maps \Rplus \to \Rplus$.
        \item[\textbf{D3:}] (posetal structure) Give $\F\Rplus = \Rplus$ the same ordering.
        \item[\textbf{D4:}] (comparison lemma) Let $\gamma \maps \N \to \Rplus$ be a sequence of positive reals that makes the left diagram of Axiom \textbf{D4} lax commute, i.e.
        \[\begin{tikzcd}
            \N
            \arrow[r, "\gamma"]
            &
            \Rplus
            \\
            \N
            \arrow[u, "+1"]
            \arrow[r, "\gamma"']
            &
            \Rplus
            \arrow[u, "\id_\Rplus"']
            \arrow[ul, Rightarrow]
        \end{tikzcd}\]
        lax commutes. This says precisely that $\gamma(n+1) \leq \gamma(n)$ for each $n \in \N$. It is then clear that $\gamma(n) \leq \gamma(0)$ for all $n \in \N$, which is what the right diagram of Axiom \textbf{D4} says.
    \end{itemize}
    The coalgebraic decrescent condition for a function $V \maps E \to \Rplus$ to be Lyapunov according to \cref{thm:LyapTrajectory} says that $V(F(x)) \leq V(x)$ for all $x \in E$, which again recovers the classical Lyapunov condition for discrete-time systems. 

    Given a flow $\flow \maps \N \times E \to E$, then $c_k := \flow_k(x_0)$ is a solution to $F \maps E \to E$ if the diagram commutes: 
    \begin{equation}
    \begin{tikzcd}
        \N
        \arrow[r, "c"]
        &
        E 
        \\
        \N
        \arrow[u, "+1"]
        \arrow[r, "c"']
        &
        E
        \arrow[u, "F"']
    \end{tikzcd}
    \end{equation}
    which simply states that $c_{k+1} = F(c_k)$. A point $x^*$ is an equilibrium point if $c_k = x^*$ for all $k$, therefore $c_1 = F(c_0) = F(x^*) = x^*$, i.e., $x^*$ is a fixed point of $F$. This equilibrium point is stable if $\| c(k) \|_{x^*} \leq \alpha(\| c(0) \|_{x^*})$. This is equivalent to: 
    for all $\epsilon > 0$, there exists a $\delta > 0$ such that: 
    \[
    \| c(0) - x^* \| < \delta \qquad \mathrm{implies} \qquad 
    \| c(k) = F^k(x) - x^* \| < \epsilon  \quad \forall k \in \N.
    \]
    
    For $V \maps E \to \Rplus$ to be a Lyapunov function, according to \cref{thm:lyapunovfsys}, $V$ must be positive definite and the following diagram must lax commute:  
    \[
    \begin{tikzcd}
        E 
        \arrow[r, "V"]
        &
        \Rplus 
        \\
        E 
        \arrow[u, "F"]
        \arrow[r, "V"']
        &
        \Rplus
        \arrow[u, "id"']
        \arrow[ul, Rightarrow, ""]
    \end{tikzcd}
    \]
    Which simply yields: 
    \[
    V(F(x)) \leq V(x) \quad \textrm{implies} \quad 
    \nabla V(x) = V(F(x))  - V(x) \leq 0 
    \]
    and we arrive at the classic Lyapunov conditions for discrete-time dynamical systems. 
\end{example}

\begin{example}[Graphs]
\label{ex:graphs2}
    Recall from \cref{ex:graphs} that we can interpret a coalgebra $f \maps A \to \P(A)$ of the covariant powerset functor on $\Set$ as a graph, where $A$ is the vertex set, and $f(a)$ is the collection of vertices with an edge pointing to it from the vertex $a$.
    \begin{itemize}
        \item[\textbf{D0:}] (setting) Use the setting for stability given in \cref{ex:discrete1}
        \item[\textbf{D1:}] (unit clock) We give the time object $T = \N$ the unit clock graph $ 1 = +1 \maps \N \to \P(\N)$ where $n \mapsto \{n+1\}$. 
        \item[\textbf{D2:}] (stable system) Give $R = \Rplus$ the graph $\Rplus \to \P(\Rplus)$ given by $r \mapsto \{r\}$.
        \item[\textbf{D3:}] (posetal structure) Give $\P(\Rplus)$ the partial order given by $U \leq V$ iff $u \leq v$ for all $u \in U$ and $v \in V$. 
        \item[\textbf{D4:}] (comparison lemma) Let $\gamma \maps \N \to \Rplus$ be a sequence of positive reals such that the left diagram of Axiom \textbf{D4} lax commutes. This translates to $\{\gamma_{n+1}\} \leq \{\gamma_n\}$ which is precisely $\gamma_{n+1} \leq \gamma_n$, i.e., $\gamma$ is a decreasing sequence. It is clear then that $\gamma_0 \geq \gamma_n$, which is what the right diagram of Axiom \textbf{D4} says. 
    \end{itemize}
    In this setting, a trajectory is a sequence of vertices $\gamma \maps \N \to A$ such that for each $n \in \N$, the set of vertices pointed to by the vertex $\gamma_n$ is precisely $\{\gamma_{n+1}\}$. This indicates that this framework is limited to discussing ``deterministic'' graphs, those where each vertex has out-degree exactly 1. One could consider lax maps of coalgebras to allow more general notions of graph maps, but this does not serve our current purposes, and so lies outside the scope of the present work. 
\end{example}

\begin{example}[Monoid-Labeled Transition Systems]
\label{ex:monoidLTS}
    Let $(T, \oplus, 0)$ be a commutative monoid in $\Set$. Then we can give the following setting for dynamic stability.
    \begin{itemize}
        \item[\textbf{D0:}] (setting) Use the setting for stability given by using $T$ as the time object, and $R = \Rplus$. 
        \item[\textbf{D1:}] (unit clock) Equip the time object $T$ with the system $1_T \maps T \to \P(T \times T)$ which sends $t \in T$ to $\{(s, t+s)\}$.
        \item[\textbf{D2:}] (stable system) Define the stationary system on $R = \Rplus$ to be $0 \maps \Rplus \to \P(T \times \Rplus)$ given by $r \mapsto \{(0,r)\}$.
        \item[\textbf{D3:}] (posetal structure) Even though the powerset of any naturally carries a partial ordering by inclusion, a different ordering is more useful for our purposes. Instead, give $T \times \Rplus$ the order defined by $(t,x) \leq_2 (t',y)$ if and only if $x \leq y$, and then $\P(T \times \Rplus)$ the induced order $U \leq_{P2} V$ means that if $(u_0, u_1) \in U$ and $(v_0,v_1) \in V$, then $(u_0,u_1) \leq_2 (v_0,v_1)$.
        \item[\textbf{D4:}] (comparison lemma) Let $\gamma \maps T \to R$ be a map making the left diagram of Axiom \textbf{D4} lax commute. This means that 
        \begin{align*}
            \{(r, \gamma_{t+r})\}
            &= \{(r,p) \mid \gamma_{t + r} = p\}
            \\&= \P(\id \times \gamma)(\{(r,s) \mid t + r = s\})
            \\&= \P(\id \times \gamma)(1_T(t))
            \\&\leq_{P2} 0_R(\gamma_t)
            \\&= \{(0, \gamma_t)\}.
        \end{align*}
        So then for each $t \in \Rplus$, we have $(r, \gamma_{t+r}) \leq_2 (0, \gamma_t)$, which means $\gamma_{t+r} \leq \gamma_t$. This returns the right diagram of Axiom \textbf{D4} by taking $t = 0$.
    \end{itemize}
    In \cref{ex:BehaviorLTS_cts} and \cref{ex:BehaviorLTS_discrete} below, we examine the case where $T$ is $\N$ and $\Rplus$ respectively to give settings for the behavior of systems encoded as labeled transition systems. This leads to a faux setting for hybrid systems viewed through behavior-encoded LTSs in \cref{ex:BehaviorLTS_hybrid}. To properly capture hybrid systems requires a more elaborate framework, and is beyond the scope of the present work.
\end{example}

\begin{example}[Behavioral LTS of a continuous system]
\label{ex:BehaviorLTS_cts}
    Recall \cref{ex:LTS1}. Consider a classical continuous-time dynamical system $\dot x = f(x)$ on a space $E$ which is complete in the sense that each initial condition $x_0$ admits a solution curve $\Rplus \to E$. These solutions taken together form an action $\flow \maps \Rplus \times E \to E$ of the monoid $(\Rplus,+)$ on $E$, called the flow. We can encode the flow as an $\Rplus$-labeled transition system as follows. The underlying set of the space $E$ is the state space, and $x \xrightarrow t x'$ if $\flow(x,t) = x'$. 
    This LTS $B_f \maps E \to \P(\Rplus \times E)$ is called the \define{behavior of the system} \cite{tabuada2009verification}.
    \begin{enumerate}
    \setcounter{enumi}{-1}
        \item[\textbf{D0:}] (setting) Take the setting for stability given in \cref{ex:classicstability1}. 
        \item[\textbf{D1:}] (unit clock) Equip the time object $T = \Rplus$ with the unit clock $1_T \maps T \to \P(T \times T)$ given by $t \mapsto \{(r,s) \in T \times T \mid t+r=s\}$.
        \item[\textbf{D2:}] (stable system) Define the stationary system on $R = \Rplus$ to be $0 \maps \Rplus \to \P(\Rplus \times \Rplus)$ given by $r \mapsto \{(0,r)\}$.
        \item[\textbf{D3:}] (posetal structure) Give $\Rplus^2$ the order defined by $(a,b) \leq_2 (c,d)$ if and only if $b \leq d$, and then $\P(\Rplus^2)$ the induced order $U \leq_{P2} V$ means that if $(u_0, u_1) \in U$ and $(v_0,v_1) \in V$, then $(u_0,u_1) \leq_2 (v_0,v_1)$.
        \item[\textbf{D4:}] (comparison lemma) Let $\gamma \maps T \to R$ be a map making the left diagram of Axiom \textbf{D4} lax commute. This means that 
        \begin{align*}
            \{(r, \gamma_{t+r})\}
            &= \{(r,p) \mid \gamma_{t + r} = p\}
            \\&= \P(\id \times \gamma)(\{(r,s) \mid t + r = s\})
            \\&= \P(\id \times \gamma)(1_T(t))
            \\&\leq_{P2} 0_R(\gamma_t)
            \\&= \{(0, \gamma_t)\}.
        \end{align*}
        So then for each $t \in \Rplus$, we have $(r, \gamma_{t+r}) \leq_2 (0, \gamma_t)$, which means $\gamma_{t+r} \leq \gamma_t$. This returns the right diagram of Axiom \textbf{D4} by taking $t = 0$.
    \end{enumerate}
    In this setting, the coalgebraic decrescent condition \[\P(\Rplus \times V) \circ \B_f \leq 0_R \circ V\] translates to 
    \[\P(\Rplus \times V) (B_f(x)) \leq \{(0, V(x))\}\]
    where we can write $B_f(x) = \{(t, x') \mid \flow_t(x)=x'\}$, 
    \begin{align*}
        \{(0, V(x))\}
        &\geq \P(\Rplus \times V) (B_f(x)) 
        \\&= \P(\Rplus \times V) (\{(t, x') \mid \flow_t(x)=x'\})
        \\&= \{(t, r) \in \Rplus^2\mid V(\flow_t(x))=r\}
        \\&= \{(t, V(\flow_t(x)))\}.
    \end{align*}
    By the definition of the order on $\P(\Rplus \times \Rplus)$, the above inequality reduces to saying $V(\flow_t(x)) \leq V(x)$ for all $t \in \Rplus$. 
\end{example}

\begin{example}[Behavioral LTS of a discrete system]
\label{ex:BehaviorLTS_discrete}
    Similar to the previous example, we can package the data of a discrete-time system as an LTS as well. 
    \begin{itemize}
        \item[\textbf{D0:}] (setting) We use the setting for stability for discrete-time systems given in \cref{ex:discrete1}.
        \item[\textbf{D1:}] (unit clock) Equip the time object $T=\N$ with the $\N$-LTS $1 \maps \N \to \P(\N \times \N)$ given by $n \mapsto \{((m,n+m))\}$
        \item[\textbf{D2:}] (stable system) Equip $R =\Rplus$ with the system $0 \maps \Rplus \to \P(\N \times \Rplus)$ given by $r \mapsto \{(0, r)\}$. 
        \item[\textbf{D3:}] (posetal structure) Give $\N \times \Rplus$ the order defined by $(a,b) \leq_2 (c,d)$ if and only if $b \leq d$, and then given $\P(\N \times \Rplus)$ the induced order $U \leq_{P2} S$ means that if $(u_0,u_1) \in U$ and $(v_0,v_1) \in V$, then $(u_0,u_1) \leq_2 (v_0,v_1)$. 
        \item[\textbf{D4:}] (comparison lemma) Let $\gamma \maps \N \to \Rplus$ be a map making the left diagram of Axiom \textbf{D4} lax commute. This means that 
        \begin{align*}
            \{(m,\gamma_{n+m}\}
            &= \P(\N \times \gamma) \{(m,n+m)\}
            \\&= \P(\N \times \gamma) (1(n))
            \\&\leq \sigma(\gamma_n)
            \\&= \{(0,\gamma_n)\}.
        \end{align*}
        So then for each $n \in \N$, we have $(m,\gamma_{n+m}) \leq_2 (0, \gamma_n)$, which means $\gamma_{n+m} \leq \gamma_n$. This returns the right diagram of Axiom \textbf{D4} by taking $n = 0$.
    \end{itemize}
\end{example}

\begin{example}[Behavioral LTS of a hybrid system]
\label{ex:BehaviorLTS_hybrid}
    The behavior of a hybrid system can also be expressed as an LTS \cite{tabuada2009verification}. Hybrid time cannot be expressed using monoids and actions in the way we have in \cite{CLT1} and the present work. It fundamentally requires a different framework to capture the correct notion of flow. However, there are still several aspects of the theory we can capture on the coalgebraic side of the story. Consider a hybrid system which has a set $X$ of states, a set $A = U \sqcup \Rplus$ of transition labels, and a transition relation $\to \subseteq X \times A \times X$. In this setup, the elements of the set $U$ are used to label the transitions that occur when the system makes a discrete jump. Thus, we will consider the hybrid system as an $A$-LTS, i.e., a $\P(A \times -)$-system.
    \begin{itemize}
        \item[\textbf{S0:}] (setting) The category is $\Set$.
        \item[\textbf{S1:}] (space) The set of states $X$ is our space of interest.
        \item[\st{\textbf{S2:}}] \st{(time)} As noted, the expression of time does not fit into the monoid action framework. We will nonetheless take $T = A$ for now.
        \item[\textbf{S3:}] (stable object) We will still use $\Rplus$ as the stable object.
        \item[\textbf{S4:}] (distance) We use the standard metric for $d$.
        \item[\st{\textbf{D1:}}] \st{(unit clock)}
        As we do not have a time object, we cannot equip it with a coalgebra.
        \item[\textbf{D2:}] (stable system) As the set $A$ does contain $\Rplus$ as a subset, we can follow the previous examples and define $0 \maps \Rplus \to \P(A \times \Rplus)$ by $r \mapsto \{(0,r)\}$. 
        \item[\textbf{D3:}] (posetal structure) Similar to the previous examples, we can define an order $\leq_2$ on $A \times \Rplus$ by only comparing the second components, and then use this to induce an order $\leq_{P2}$ on $\P(A \times \Rplus)$. 
        \item[\st{\textbf{D5:}}] \st{(comparison lemma)} This axiom also depends on the time object, and so we cannot consider it.
    \end{itemize}
    This indicates the need for a modification of our framework which treats time in a more general way, beyond monoid actions. We leave this for future work.
\end{example}

\section{Monoidal Settings for Dynamic Stability}
\label{sec:monoidal}

Lyapunov theory for systems can be formulated building upon the case of flows \cite{CLT1}. This requires additional axioms for systems. These axioms completely characterize the stability of systems, as we derive necessary and sufficient conditions for the Lyapunov stability of systems. We begin by proving sufficiency, which holds essentially for formal reasons, followed by the converse, which requires substantially more careful treatment. 

\begin{defn}
\label{def:monoidalsettingdynamicstabilty}
    A \define{monoidal setting for dynamic stability} is a setting for dynamic stability $T \into E \onto R$ in $\C$ such that
    \begin{itemize}
        \item[\textbf{D5:}] (laxator) There is a natural transformation with components $\psi_{A,B} \maps \F A \times \F B \to \F(A \times B)$ such that the following diagram commutes.
        \[\begin{tikzcd}
            (\F X \times \F Y) \times \F Z
            \arrow[r, "\alpha_{\F X, \F Y, \F Z}"]
            \arrow[d, "\psi_{X,Y} \times \id"']
            &
            \F X \times (\F Y \times \F Z)
            \arrow[r, "\id \times \psi_{Y,Z}"]
            &
            \F X \times \F(Y \times Z)
            \arrow[d, "\psi_{X,Y \times Z}"]
            \\
            \F(X \times Y) \times \F Z
            \arrow[r, "\psi_{X \times Y, Z}"']
            &
            \F((X \times  Y) \times Z)
            \arrow[r, "\F\alpha_{X,Y,Z}"']
            &
            \F(X \times (Y \times Z))
        \end{tikzcd}\]
        
        \item[\textbf{D6:}] (stationary systems) There is a natural transformation with components $0_A \maps A \to \F A$ such that the following diagram commutes.
        \[\begin{tikzcd}[column sep = small]
            A \times B 
            \arrow[rr, "0_A \times 0_B"]
            \arrow[dr, "0_{A \times B}"']
            &&
            \F A \times \F B
            \arrow[dl, "\psi_{A,B}"]
            \\&\F(A \times B)
        \end{tikzcd}\]
        
        \item[\textbf{D4':}] (generalized comparison lemma) ${0}_R \maps R \to \F R$ is a stable system and for any $A \in \C$: 
        \[
        \begin{tikzcd}
            T \times A
            \arrow[d, "\L_A"']
            \arrow[r, "\gamma"]
            &
            R
            \arrow[d, "{0}_R"]
            \arrow[dl, Rightarrow, ""]
            \\
            \F(T \times A)
            \arrow[r, "\F \gamma"']
            &
            \F R
        \end{tikzcd}
        \qquad \mathrm{implies} \qquad 
        \begin{tikzcd}
            A
            \arrow[r, "0_\oplus \times \id_A", ""'{name = A}]
            &
            T \times A
            \arrow[d, "\gamma"]
            \\
            T \times A
            \arrow[u, "\pi_A"]
            \arrow[r, "\gamma"', ""{name = B}]
            &
            R
            \arrow[from = A, to = B, Rightarrow]
        \end{tikzcd}
        \]
        where the coalgebra $\L_X \maps T \times X \to \F(T \times X)$ is defined to be the composite 
        \begin{equation}
        \label{eq:Ldefinition}
        \L_X \coloneq \left[ T \times X \xrightarrow{1_T \times 0_X} \F T \times \F X \xrightarrow{\psi_{T,X}} \F(T \times X)\right].
        \end{equation}
    \end{itemize}
\end{defn}

\begin{example}[Manifolds]
\label{ex:manifolds4}
    Recall that in \cref{ex:classicstability1} we gave a setting for stability in $\Set$ associated to an ODE $\dot x = f(x)$, and in \cref{ex:classicstability2} we gave a setting for dynamic stability in $\Man$. It is important that the underlying category is $\Man$ for dynamic stability because this is required to form a tangent bundle functor for which vector fields are (a special case of) coalgebras.
    \begin{itemize}
        \item[\textbf{D5:}] (laxator) The natural transformation $\psi_{M,N} \maps \T M \times \T N \to \T(M \times N)$ takes a tangent vector on $M$ and a tangent vector on $N$, and realizes the pairing of them is a tangent vector on the product manifold $M \times N$. In that case, $\psi$ is invertible, but this is not generally true. 
        \item[\textbf{D6:}] (stationary systems) The stationary system on a manifold $M$ is the constant zero vector field $0_M \maps M \to \T M$.
        \item[\textbf{D4':}] (generalized comparison lemma) The generalized comparison lemma follows from the classical comparison lemma in a pointwise fashion. The vector field ${0} \maps \Rplus \to \Rplus \times \R$ given by ${0}(r) = (r,0)$. In this setting, axiom \textbf{D4'} becomes: 
        \[
        \begin{tikzcd}
            \Rplus \times \R
            \arrow[r, "T \gamma"]
            &
            \Rplus \times \R
            \\
            \Rplus
            \arrow[u, "{1}"]
            \arrow[r, "\gamma"']
            &
            \Rplus
            \arrow[u, "{0}"']
            \arrow[ul, Rightarrow, ""]
        \end{tikzcd}
        \qquad \mathrm{implies} \qquad 
        \begin{tikzcd}
            \Rplus
            \arrow[r,bend left=35pt,"\gamma_0", ""'{name = A}]
            \arrow[r,bend right=35pt,"\gamma"', ""{name = B}]
            &
            \Rplus
            \arrow[from = A, to = B, Rightarrow]
        \end{tikzcd}
        \] 
        where $T\gamma(t,1) = (\gamma(t),\dot{\gamma}(t))$, i.e., 
        \[
        \dot{\gamma}(t) \leq 0 \qquad \textrm{implies} \qquad \gamma(t) \leq \gamma(0)
        \]
        which is a special case of the classical comparison lemma. 
    \end{itemize}
\end{example}

\begin{example}[Discrete-time stability]
\label{ex:discrete4}
    Recall the setting for discrete-time stability given in \cref{ex:discrete3}.
    \begin{itemize}
        \item[\textbf{D5:}] (laxator) Let the laxator be $\id_{A \times B} \maps A \times B \to A \times B$.
        \item[\textbf{D6:}] (stationary systems) The stationary system on an object $A$ is $\id_A \maps A \to A$. 
        \item[\textbf{D4':}] (generalized comparison lemma) The generalized comparison lemma follows from the ordinary comparison lemma seen in \cref{ex:discrete3} in a pointwise fashion.
    \end{itemize}
\end{example}

\begin{example}[Labeled Transition Systems]
\label{ex:LTS3}
    Recall \cref{ex:BehaviorLTS_cts}.
    \begin{itemize}
    \item[\textbf{D5:}] (laxator) Let the laxator be given by \[
    \psi_{X,Y}:s\times 's\mapsto \{(a,x,y)\mid (a,x)\in s \text{ and } (a,y)\in s'\}.\]
    \item[\textbf{D6:}] (stationary systems) The stationary system $0_X \maps X \to P(A\times X)$ given by $x\mapsto \{(a,x)\mid a\in A\}=A\times \{x\}$.
    \item[\textbf{D4':}] (generalized comparison lemma) The generalized comparison lemma follows from the comparison lemma seen in \cref{ex:BehaviorLTS_cts}.
\end{itemize}
\end{example}

\begin{rmk}
\label{rmk:tensorproduct}
    The laxator provides a way of combining systems. Given two $\F$-systems, $ f \maps A \to \F A$ and $ g \maps B \to \F B$, we can construct a new system $ f \otimes_\psi  g$ as the following composite.
    \begin{equation*}
         f \otimes_\psi  g \coloneq \left[ A \times B \xrightarrow{ f \times  g} \F A \times \F B \xrightarrow{\psi_{A,B}} \F(A \times B) \right]
    \end{equation*}
    For instance, the system $\L_X$ given in \cref{eq:Ldefinition} is exactly $ 1_T \otimes_\psi 0_X$. This defines a tensor product on $\Sys_\F$ which is associative due to the assumption that the diagram in Axiom \textbf{D5} commutes. This forms the ``binary'' components of the data and properties that would make $\F$ into a lax monoidal functor, and $\Sys_\F$ into a monoidal category. We do not need the unital data and properties in this section as they do not come into play in the proof of the Lyapunov theorem. We discuss what the unital conditions afford us in \cref{sec:converse}.
\end{rmk}

\begin{rmk}
\label{rmk:stationarysystem}
    Stationary systems ${0}_A \maps A \to \F A$ give to each object of $\C$ a system which should be interpreted as ``standing still'', depending on the context. The naturality condition says that every map of $\C$ is a system map between the corresponding zero systems.
    The commuting diagram in Axiom \textbf{D6} says that the zero system on a product is the same as the zero system on each component followed by $\psi$, i.e., $0_A \otimes_\psi 0_B = 0_{A \times B}$. 
\end{rmk}

\begin{defn}
\label{def:functorL}
    Recall the construction of an $\F$-system $\L_X \maps T \times X \to \F(T \times X)$ for any given object $X \in \C$ given in \cref{eq:Ldefinition}.
    It is easy to check that this construction gives the components of a natural transformation $T \times - \To \F(T \times -)$. As noted in \cref{rmk:functtocoalgisnat}, this equivalently defines a functor $\L \maps \C \to \Sys_\F$. 
\end{defn}

\begin{rmk}
\label{rmk:Tx-monad}
    The composite functor $U \circ \L$ sends an object $X \in \C$ to $T \times X$. The functor $T \times - \maps \C \to \C$ induced by a monoid object $T$ carries the structure of a monad. The monad multiplication $\mu \maps T \times T \times - \To T \times -$ has components $\mu_X \maps T \times T \times X \to T \times X$ given by $\oplus \times \id_X$, and monad unit $\eta \maps \id_\C \To T \times -$ has components $\eta \maps X \to T \times X$ given by $0 \times \id_X$. Algebras of this monad are flows of $T$, and the Eilenberg-Moore category is $\C^{T \times-} \simeq \TFlow$. 
\end{rmk}

\begin{defn}
\label{def:solution}
    Let ${f}_E \maps E \to \F E$ be an $\F$-system, and let $A \in \C$. A morphism $\flow \maps T \times A \to E$ is a \define{solution} to ${f}_E$ if $\flow$ is a morphism of $\F$-systems $\L_A \to {f}_E$, i.e., if the following diagram commutes.
    \[
    \begin{tikzcd}
        T \times A
        \arrow[r, "\flow"]
        \arrow[d, "\L_A"']
        &
        E
        \arrow[d, "{f}_E"]
        \\
        \F(T \times A)
        \arrow[r, "\F\flow"']
        &
        \F E
    \end{tikzcd}
    \exmath{\frac{\partial}{\partial t} \flow_t(x) \mid_{x=\flow_t(x_0)} = {f}_E(\flow_t(x_0))}
    \]
    The flow $\flow$ is called a \define{solution flow} if $A=E$ and $\flow$ satisfies the conditions to be a $T$-flow.
\end{defn}

We can now present Lyapunov's theorem for systems---extending Lyapunov's theorem for flows (cf. \cref{thm:Lyap}). In particular, we will provide conditions for when a morphism $V \maps E \to R$ is a Lyapunov morphism via a lax commuting diagram of systems. Per \cref{def:lyapunov}, recall that a Lyapunov morphism must satisfy two conditions: (1) positive definite and (2) decrescent on flows. 

\begin{thm}[Lyapunov's theorem for systems]
\label{thm:lyapunovfsys}
   Let ${f}_E \maps E \to \F E$ be a system in a setting for dynamic stability. If $V \maps E \to R$ is a morphism in $\C$ that is positive definite with respect to an equilibrium point $x^* \maps 1 \to E$, and the following diagram lax commutes: 
    \[
    \begin{tikzcd}
        E
        \arrow[r, "V"]
        \arrow[d, "{f}_E"']
        &
        R
        \arrow[d, "{0}_R"]
        \arrow[dl, Rightarrow]
        \\
        \F E
        \arrow[r, "\F V"']
        &
        \F R
    \end{tikzcd}
    \exmath{\frac{\partial V}{\partial x}  {f}_E(x) \leq 0}\]
    then $V$ is a Lyapunov morphism for every solution $\flow \maps T \times E \to E$ of ${f}_E$. Therefore, the point $x^*$ is a stable equilibrium point if a solution flow exists. 
\end{thm}
\begin{proof}
    For every flow $\flow$ that is a solution to ${f}_E$, consider the following lax commuting diagram.
    \[
    \begin{tikzcd}
        T \times E 
        \arrow[r, "\flow"]
        \arrow[d, "\L_E"']
        &
        E
        \arrow[r, "V"]
        \arrow[d, "{f}_E"']
        &
        R
        \arrow[d, "{0}_R"]
        \arrow[dl, Rightarrow]
        \\
        \F(T \times E)
        \arrow[r, "\F \flow"']
        &
        \F E 
        \arrow[r, "\F V"']
        &
        \F R
    \end{tikzcd}
    \]
    The comparison lemma (see \textbf{D4}) implies that the following diagram lax commutes.
    \[\begin{tikzcd}[column sep = small]
    	T \times E & E \\
    	& T \times E & E \\
    	E && R
    	\arrow["\pi", from=1-1, to=1-2]
    	\arrow["\flow"', from=1-1, to=3-1]
    	\arrow["{0_\oplus \times \id_E}"', from=1-2, to=2-2]
    	\arrow["\id_E"{description}, from=1-2, to=2-3]
    	\arrow["\flow"', from=2-2, to=2-3]
    	\arrow[Rightarrow, from=2-2, to=3-1]
    	\arrow["V", from=2-3, to=3-3]
    	\arrow["V"', from=3-1, to=3-3]
    \end{tikzcd}\]
    Thus $V$ is decrescent, i.e., a Lyapunov morphism. The fact that $x^* \maps 1 \to E$ is stable for every flow that is a solution to ${f}_E$ follows directly from \cref{thm:Lyap}. 
\end{proof}

Applying \cref{thm:lyapunovfsys} to \cref{ex:manifolds4}, we recover \cref{thm:ClassicLyapunov}, the classic Lyapunov theorem for dynamical systems.

\begin{example}[Markov kernels]
\label{ex:MarkovKernels2}
    Recall from \cref{ex:MarkovKernels1} that Markov kernels are coalgebras of the endofunctor $D \maps \Meas \to \Meas$ which sends a measure space $X$ to the measure space $D(X)$ of probability measures on $X$. 
    \begin{itemize}
        \item[\textbf{D0:}] (setting) Take the classical discrete-time setting \cref{ex:discrete1}.
        \item[\textbf{D1:}] (unit clock) The unit clock $\N \to D(\N)$ is given by $n \mapsto \delta_{n+1}$.
        \item[\textbf{D3:}] (posetal structure) Define a partial order $\leq$ on $D(\Rplus)$ by saying that $\mu \leq \nu$ if $\mu([0,r)) \leq \nu([0,r))$ for all $r \in \Rplus$.
        \item[\textbf{D5:}] (laxator) $\psi \maps D(X) \times D(Y) \to D(X \times Y)$ sends a pair of probability measures $(\mu, \nu)$ to the product measure $\mu \cdot \nu$ on $X \times Y$. 
        \item[\textbf{D6:}] (stationary systems) $0_X \maps X \to D(X)$ by $x \mapsto \delta_x$.
        \item[\textbf{D4':}] (generalized comparison lemma) Let $\gamma \maps \N \times A \to \Rplus$ be a map that makes the left diagram of Axiom \textbf{D4'} lax commute. By definition, $\L_A \maps \N \times A \to D(\N \times A)$ is exactly $(n,a) \mapsto \delta_{n+1} \cdot \delta_a = \delta_{(n+1,a)}$.
        \[
        \begin{tikzcd}
            \N \times A
            \arrow[d, "{(n,a)} \mapsto\delta_{(n+1,a)}"']
            \arrow[r, "\gamma"]
            &
            \Rplus
            \arrow[d, "r \mapsto \delta_r"]
            \arrow[dl, Rightarrow]
            \\
            D(\N \times A)
            \arrow[r, "D\gamma"']
            &
            D\Rplus
        \end{tikzcd}\]
        The top-right path sends $(n,a)$ to the probability measure $\delta_{\gamma(n,a)}$, and the left-bottom path sends $(n,a)$ to $\delta_{\gamma(n+1,a)}$. So the diagram above lax commuting implies that $\gamma(n+1,a) \leq \gamma(n,a)$, i.e., $\gamma$ is decreasing in the first component. Thus all the values $\gamma(-,a)$ takes are bounded above by $\gamma(0,a)$, which is precisely what the right diagram of Axiom \textbf{D4'} asks.
        \[
        \begin{tikzcd}
            A
            \arrow[r, "0_\oplus \times \id_A", ""'{name = A}]
            &
            \N \times A
            \arrow[d, "\gamma"]
            \\
            \N \times A
            \arrow[u, "\pi_A"]
            \arrow[r, "\gamma"', ""{name = B}]
            &
            \Rplus
            \arrow[from = A, to = B, Rightarrow]
        \end{tikzcd}
        \]
    \end{itemize}
    Trajectories $\gamma$ in a Markov kernel $m \maps E \to DE$ must make the following diagram commute.
    \[\begin{tikzcd}
        \N
        \arrow[r, "\gamma"]
        \arrow[d, "n \mapsto \delta_{n+1}"']
        &
        E
        \arrow[d, "m"]
        \\
        D\N
        \arrow[r, "D\gamma"']
        &
        DE
    \end{tikzcd}\]
    This says that for each $n \in \N$, the probability measure assigned to $\gamma_n \in E$ must be $\delta_{\gamma_{n+1}}$. In other words, a trajectory is a sequence of completely deterministic states in $m$. To capture genuinely stochastic behaviors requires a more general framework. This lies beyond the scope of the present work.
\end{example}

\section{Existence and Uniqueness}
\label{sec:existunique}

In this section, we develop some theory around existence and uniqueness of solutions to systems in a monoidal setting for stability. 

\begin{defn}
\label{def:derivative}
    Let $\flow \maps T \times E \to E$ be a $T$-flow. The \define{derivative}, denoted by $\D\flow$, is given by the following composite.
    \[
    \D\flow := \left[ E \xrightarrow{0_\oplus \times \id_E} T \times E \xrightarrow{\L_E} \F(T \times E) \xrightarrow{\F\flow} \F E
    \right]
    \]
    where $0_\oplus$ is the neutral element of the monoid structure on $T$, given by \cref{lem:Tselfcomplete}. 
\end{defn}

Note that in the context of vector fields, we are defining the ``derivative'' of the flow $\flow$ to be the vector field given by the time derivative of $\flow$ at time $0$.
\[\D \flow(x) \coloneq \frac{\partial \flow_t(x)}{\partial t}\bigg|_{t=0}\]
The choice of time $0$ may seem arbitrary at first, but action associativity implies that if a point $x$ is in the future of some other point, i.e., $x=\flow_{t_0}(x_0)$, then \[\frac{\partial\flow_t(x)}{\partial t}\bigg|_{t=0} = \frac{\partial\flow_t(x_0)}{\partial t}\bigg|_{t=t_0}\]
which is just another way of saying that the system $\D\flow$ is time-invariant.

\begin{defn}
\label{def:Tcomplete}
    Let $T \in \C$ be an object equipped with a point $0_\oplus \maps 1 \to T$ and the unit clock $\F$-system ${1}_T \maps T \to \F T$. A \define{$T$-complete system} is an $\F$-system ${f}_A \maps A\to \F A$ such that for every $g \maps B\to A$, there is a unique morphism of $\F$-systems $\gamma \maps (T \times B,\L_B) \to (A, {f}_A)$ such that the following diagram commutes: 
    \[
    \begin{tikzcd}
        B
        \arrow[r, "0_\oplus \times \id_B"]
        \arrow[dr, "g"']
        &
        T \times B 
        \arrow[d, "\gamma", dashed]
        \arrow[r, "\L_B"]
        &
        \F(T \times B)
        \arrow[d, "\F \gamma", dashed]
        \\&
        A
        \arrow[r, "{f}_A"']
        &
        \F A
    \end{tikzcd}\]
    Let $\TComp$ be the full subcategory of $\Sys_\F$ spanned by $T$-complete systems.
\end{defn}

\begin{defn}
\label{def:integral}
    For a $T$-complete system $ f_E \maps E \to \F E$, Let $\I{f}_E$ denote the map of systems $\L_E \to  f_E$, termed the \define{integral}, that is given by applying the $T$-completeness property to the identity map, as displayed below.
    \begin{equation}
    \label{eq:integral}
    \begin{tikzcd}
        E
        \arrow[r, "0_\oplus \times \id_E"]
        \arrow[dr, "\id_E"']
        &
        T \times E 
        \arrow[d, "\I{f}_E", dashed]
        \arrow[r, "\L_E"]
        &
        \F(T \times E)
        \arrow[d, "\F \I{f}_E", dashed]
        \\&
        E
        \arrow[r, "{f}_E"']
        &
        \F E
    \end{tikzcd}
    \end{equation}
\end{defn}

\begin{thm}[Existence and Uniqueness]
\label{thm:existenceuniqueness}
    Assume that the unit clock $1_T \maps T \to \F T$ is itself $T$-complete. 
    \begin{enumerate}
        \item 
        Given a $T$-complete system ${f}_E \maps E \to \F E$, $\I {f}_E$ is a solution to ${f}_E$. 
        Moreover, the integral $\I$ of \cref{def:integral} defines a functor:
        \[
        \I \maps \TComp \to \TFlow
        \]
        with the action on morphisms defined trivially. 
        \item 
        If $\D\flow$ is $T$-complete for all $T$-flows $\flow$, then for any solution flow $\flow$ of a system ${f}_E \maps E \to \F E$, 
        \[
        \flow = \I{f}_E. 
        \]
        Moreover, the derivative defines a functor $\D \maps \TFlow \to \TComp$, satisfying: 
        \[
        \flow = \I\D \flow, \qquad {f}_E = \D \I {f}_E 
        \]
        Therefore, $\I$ is an isomorphism of categories with inverse $\D$:
        \[
        \begin{tikzcd}
            \TComp
            \arrow[r, bend left = 20, "\I"]
            \arrow[r, phantom, "\simeq"]
            &
            \TFlow
            \arrow[l, "\D", bend left = 20]
        \end{tikzcd}
        \]
    \end{enumerate}
\end{thm}

The following diagram summarizes the all the functors in this section, and their relationships under the conditions of \cref{thm:existenceuniqueness} and \cref{cor:monadic} below.
\[\begin{tikzcd}
    &
    \TComp
    \arrow[ddl, bend left = 15, "\U"]
    \arrow[ddr, bend left = 15, "\I"]
    \\\\
    \C
    \arrow[uur, phantom, sloped, "\bot"]
    \arrow[uur, bend left = 15, "\L"]
    \arrow[rr, bend left = 10, "T \times -"]
    \arrow[rr, phantom, "\bot"]
    &&
    \TFlow
    \arrow[uul, bend left = 15, "\D"]
    \arrow[uul, phantom, sloped, "\simeq"]
    \arrow[ll, bend left = 10, "\U'"]
    \arrow[from=3-1, to=3-1, loop, in=155, out=235, distance=10mm, "\F"]
\end{tikzcd}\]

\begin{cor}
\label{cor:monadic}
    In a setting with the uniqueness property, $\L$ is left adjoint to $U$, and this adjunction is monadic.
\end{cor}

The rest of this section is dedicated to proving \cref{thm:existenceuniqueness} and \cref{cor:monadic}. A sketch of how the lemmas therein contribute to this end follows. \cref{lem:Tselfcomplete} is of foundational significance. It says that if the unit clock is $T$-complete, then $T$ carries a monoid structure that is unique with respect to the property of being a solution to the unit clock coalgebra. Moreover, any other $T$-complete system is carries an action of that monoid, unique with respect to being a solution to the system.
\cref{lem:flowissolution} and its supporting lemmas says that the derivative is functorial. \cref{lem:TcompEquivalent} characterizes the subcategory $\TComp \subseteq \Sys_\F$ and support \cref{lem:integralfunctor} to show the integral is also functorial. We then show these functor are inverses on objects with \cref{lem:ID=1}.

\begin{lem}
\label{lem:Tselfcomplete}
    If the unit clock ${1}_T \maps T \to \F T$ is itself $T$-complete, then there exists a unique monoid structure on $T$ which is a solution to $ 1_T$. Further, for any $T$-complete system ${f}_E \maps E \to \F E$ there exists a unique $T$-flow on $E$ which is a solution to $ f_E$.
\end{lem}
\begin{proof}
    Define an operation $\ast$ on the hom-set $\Sys_\F(\L_A, {1}_T)$ by letting $f \ast g$ be the following composite.
    \[
        f \ast g := \left[ T \times A \xrightarrow{\id_T \times \Delta_A} T \times A \times A \xrightarrow{f \times \id_A} T \times A \xrightarrow{g} T 
        \right]
    \]
    Note that in the case $A=1$, this is simply $g \circ f$. Associativity of $\ast$ is a direct consequence of associativity of composition and coassociativity of $\Delta_A$. Projection $\pi_T \maps T \times A \to T$ is a neutral element with respect to this operation. This monoid structure is natural in $A$ due to naturality of $\Delta_A$. Thus $T$-completeness of ${1}_T$ allows us to transfer this monoid structure not just to the hom-set $\C(A,T)$, but the representable functor $\C(-,T) \maps \C\op \to \Set$. Since $\C(-,T) \times \C(-,T) \cong \C(-,T \times T)$, $T$ inherits a monoid structure by the Yoneda lemma. Uniqueness follows from the universal property. 

    Similarly, for a $T$-complete system ${f}_E \maps E \to \F E$, we define an action of the monoid $\Sys_\F(\L_A,{1}_T)$ on $\Sys_\F(\L_A,E)$ by duplication of the $A$ component, followed by composition. For precisely the same reasons as above, this corresponds to an action of $T$ on $E$, unique with respect to the property of being a map of $\F$-systems. 
\end{proof}

\begin{rmk}
\label{rmk:Tselfcomplete}
    One consequence of \cref{lem:Tselfcomplete} is that the unit clock system is ``compatible'' with the monoid structure on $T$ in the sense that the operation is a map of systems $\L_T \to  1_T$, i.e. the following diagram commutes. 
    \[\begin{tikzcd}
        \F(T \times T)
        \arrow[r, "\F\oplus"]
        &
        \F T
        \\
        T \times T
        \arrow[u, "\L_T"]
        \arrow[r, "\oplus"']
        &
        T
        \arrow[u, " 1_T"']
    \end{tikzcd}\]
    From this it is easy to check that $\D\oplus =  1_T$. 
\end{rmk}

\begin{defn}
    A monoidal setting for dynamic stability is said to have \define{compatible unit clock} if $ 1_T$ is $T$-complete. 
\end{defn}

It will be important to know when systems of the form $\L_X$ are $T$-complete. Notice that under the assumption of a compatible unit clock, given any map of the form $g \maps B \to T \times X$, the composite \[T \times B \xrightarrow{\id \times g} T \times T \times X \xrightarrow{\oplus \times \id} T \times X\] is a solution of $\L_X$. This makes some intuitive sense as the system $\L_X \maps T \times X \to \F(T \times X)$ is essentially defined to be the unit clock on the $T$-component, and stationary on the $X$-component. What we do not know is whether this is the only possible solution for $\L_X$ which has initial value $g$.

\begin{lem}
\label{lem:associatedcoalg}
    In a monoidal setting for dynamic stability with a compatible unit clock, $\oplus \times \id_X \maps T \times T \times X \to T \times X$ is a map of systems of the form $\L_{T \times X} \to \L_X$, and thus $\D(\oplus \times \id_X)=\L_X$. 
\end{lem}
\begin{proof} 
    \[
    \begin{tikzcd}[cramped, column sep=.5em, row sep=3.15em, trim right=13em]
        &
        \F(T\times T\times X)
        \ar[rr, "\F(\oplus\times \id_X)"]
        &&
        \F(T\times X)
        \\
        \F T\times \F(T\times X)
        \ar[ur, "\psi_{T,T \times X}"]
        &&
        |[xshift=-7.5em]|
        \F(T\times T)\times \F X 
        \ar[ul, "\psi_{T \times T, X}"]
        \ar[r, "\F\oplus\times \F\id_X"]
        &
        \F T \times \F X 
        \ar[u, "\psi_{T,X}"]
        \\
        &
        \F T \times \F T \times \F X 
        \ar[ul, "\id \times \psi_{T,X}"']
        \ar[ur, "\psi_{T,T} \times \id"]
        \\
        T\times X
        \ar[r, "0_\oplus \times \id_{T\times X}"]
        &
        T \times T \times X
        \arrow[uul, "{1}_T \times {0}_{T \times X}"]
        \ar[rr, "\oplus \times \id_X"']
        \ar[u, "{1}_T \times {0}_T \times {0}_X"']
        &&
        T\times X
        \ar[uu, "{1}_T\times {0}_X"]
    \end{tikzcd}\]
    The bottom square commutes component-wise; the left diamond is associativity of the laxator, and the top square is naturality of the laxator. Now the bottom-right path collapses to $\L_X$, the top-left path is the definition of associated system.
\end{proof}

It is immediate from the definition of $\D$ that if $\flow$ is a solution flow of $ f$, then $\D\flow =  f$. 

\begin{lem}
\label{lem:flowissolution}
    Assume a setting with compatible unit clock.
    Every $T$-flow $\flow \maps T \times E \to E$ is a solution to its derivative, i.e., $\flow$ is a map of $\F$-systems of the form $\L_E \to \D\flow$. 
\end{lem}
\begin{proof}
    Consider the following diagram.
    \[
    \begin{tikzcd}
        \D(\oplus\times \id_E) = \hspace{-1.2cm} &  \left[ T\times E \right. 
        \arrow[r, "0_\oplus \times \id"]
        \arrow[d, "\flow"']
        &
        T \times T\times E
        \arrow[r, "\L_{T \times E}"]
        \arrow[d, "\id \times \flow"']
        &
        \F(T \times T\times E) 
        \arrow[r, "\F \oplus \times \id_E"]
        \arrow[d, "\F(\id \times \flow)"']
        & 
        \left. \F(T\times E) \right]
        \arrow[d, "\F \flow"]
        \\
        \D\flow = \hspace{-2cm} & \left[ E \right.
        \arrow[r, "0_\oplus \times \id_E"']
        &
        T \times E
        \arrow[r, "\L_E"']
        & 
        \F(T \times E) 
        \arrow[r, "\F \flow"']
        &
       \left. \F E \right]
    \end{tikzcd}
    \]
    The left square commutes trivially, the middle square by functoriality of $\L$, and right square commutes by Functoriality of $\F$ (this is the defining square of $T$-flow sent trough $\F$). Recall that $\D(\oplus \times \id_E)=\L_E$ by \cref{lem:associatedcoalg}, and so $\flow$ is a map of systems $\D(\oplus \times \id_E)=\L_E \to \D\flow$.
\end{proof}

\begin{defn}
    We say that an $\F$-system $ f_E \maps E \to \F E$ has the \define{uniqueness property} if whenever $\gamma_1, \gamma_2 \maps T \times B \to E$ are system maps $\L_B \to  f_E$ such that the diagram 
    \[
    \begin{tikzcd}
        B
        \arrow[r, "0_\oplus \times \id"]
        \arrow[d, "0_\oplus \times \id"']
        &
        T \times B 
        \arrow[d, "\gamma_1"]
        \\
        T \times B
        \arrow[r, "\gamma_2"']
        &E
    \end{tikzcd}
    \]
    commutes, then $\gamma_1 = \gamma_2$.
\end{defn}
The uniqueness property is so named because it guarantees that two solutions with the same initial conditions will not diverge, that is solution are unique with respect to their initial condition.
\begin{lem}
\label{lem:TcompEquivalent}
    Assume a setting with compatible unit clock. 
    Let $\flow \maps T \times E \to E$ be a $T$-flow. The following are equivalent:
    \begin{enumerate}
        \item If $g \maps T \times B \to E$ is a system map $\L_B \to \D\flow$, then $g$ is a flow map $(T \times B, \oplus \times \id_B) \to (E, \flow)$. 
        \item $\D\flow$ has the uniqueness property.
        \item $\D\flow$ is $T$-complete.
    \end{enumerate}
\end{lem}
\begin{proof}
    ($1 \implies 2$) 
    Let $\gamma_1, \gamma_2 \maps T \times B \to E$ be solution curves with the same initial condition, i.e. system maps of type $\L_B \to \D\flow$ such that $\gamma_1 \circ (0_\oplus \times \id) = \gamma_2 \circ (0_\oplus \times \id)$. The diagram below shows $\gamma_1 = \gamma_2$.
    \[\begin{tikzcd}
    	& {T \times T \times B} & {T \times B} \\
    	{T\times B} && {T \times E} & E \\
    	& {T \times T \times B} & {T \times B}
    	\arrow["{\oplus \times id}"{description}, from=1-2, to=1-3]
    	\arrow["{id \times \gamma_1}"{description}, from=1-2, to=2-3]
    	\arrow["{\gamma_1}"{description}, from=1-3, to=2-4]
    	\arrow["{id \times 0 \times id }"{description}, from=2-1, to=1-2]
    	\arrow["id"{description}, bend left = 60, from=2-1, to=1-3]
    	\arrow["id"{description}, bend right = 60, from=2-1, to=3-3]
    	\arrow["{id \times 0 \times id}"{description}, from=2-1, to=3-2]
    	\arrow["\flow"{description}, from=2-3, to=2-4]
    	\arrow["{id \times \gamma_2}"{description}, from=3-2, to=2-3]
    	\arrow["{\oplus \times id}"{description}, from=3-2, to=3-3]
    	\arrow["{\gamma_2}"{description}, from=3-3, to=2-4]
    \end{tikzcd}\]
    The top-right and bottom-right diamonds commute by the assumption that $\gamma_1$ and $\gamma_2$ are system maps $\L_B \to \D\flow$ and condition (1). The left diamond commutes due to the assumption that $\gamma_1 \circ (0_\oplus \times \id) = \gamma_2 \circ (0_\oplus \times \id)$.

    ($2 \implies 3$) 
    Let $g \maps B \to E$ be a map in $\C$. The composite $\flow \circ (\id \times g)$ is a system map of the form $\L_B \to \D\flow$ because the following diagram commutes.
    \[\begin{tikzcd}[column sep = small]
    	&& 
        E
        \arrow[r, "0_\oplus \times \id"]
        & 
        T \times E 
        \arrow[r, "\L_E"]
        & 
        \F(T \times E)
        \arrow[dr, "\F\flow"]
        \\& 
        T \times E 
        \arrow[ur, "\flow"]
        \arrow[r, "0_\oplus \times \id"]
        \arrow[drr, "\id"description]
        & 
        T \times T \times E 
        \arrow[ur, "\id \times \flow"description] 
        \arrow[r, " \L_{T \times E}"] 
        \arrow[dr, "\oplus \times \id"description]
        & 
        \F(T \times T \times E)
        \arrow[ur, "\F(\id \times \flow)"description]
        \arrow[dr, "\F(\oplus \times \id)"description]
        &&
        \F E
        \\
    	T \times B
        \arrow[ur, "\id \times g"]
        \arrow[rrr, "\id \times g"]
        \arrow[drr, "\L_B"']
        &&&
        T \times E
        \arrow[r, "\L_E"]
        &
        \F(T \times E)
        \arrow[ur, "\F\flow"']
        \\&& 
        \F(T \times B)
        \arrow[urr, "\F(\id \times g)"']
    \end{tikzcd}\]
    Many of the cells in this diagram commute for trivial reasons.
    The square
    \[\begin{tikzcd}
        T \times T \times E
        \arrow[r, "\L_{T \times E}"]
        \arrow[d, "\oplus \times \id"']
        &
        \F(T \times T \times E)
        \arrow[d, "\F(\oplus \times \id)"]
        \\
        T \times E
        \arrow[r, "\L_E"']
        &
        \F(T\times E)
    \end{tikzcd}\]
    commutes by \cref{lem:associatedcoalg}.
    The uniqueness property then says $g$ is uniquely determined by its initial condition, and thus $\D\flow$ is $T$-complete.

    ($3 \implies 1$)
    Recall that $\oplus \maps T \times T \to T$ is a system map $\L_T \to {1}_T$ by \cref{lem:Tselfcomplete} and \cref{rmk:Tselfcomplete}, and $\flow \maps T \times E \to E$ is a system map $\L_E \to \D\flow$. The map $id \times \gamma$ is a system map $\L_{T \times B} \to \L_E$ automatically. Thus both $\gamma \circ (\oplus \times \id)$ and $\flow \circ (\id \times \gamma)$ are system maps of the form $L(T \times B) \to \D\flow$. They also both give $\gamma$ when pre-composed with $0\times id_{T \times B}$. 
    \[\begin{tikzcd}
        T \times B
        \arrow[r, "0 \times \id"]
        \arrow[dr, "\gamma"']
        &T \times T \times B
        \arrow[d, dashed, "\exists!"]
        \\&
        E
    \end{tikzcd}\]
    By universal property, the two maps must be equal.
    \[
    \begin{tikzcd}
        T \times T \times B
        \arrow[r, "\id \times \gamma"]
        \arrow[d, "\oplus \times \id"']
        &
        T \times E 
        \arrow[d, "\flow"]
        \\
        T \times B
        \arrow[r, "\gamma"']
        &
        E
    \end{tikzcd}
    \]
    Thus $\gamma$ is a $T$-flow map $(T\times B, \oplus\times \id_B) \to (E,\flow)$.
\end{proof}

\begin{lem}
\label{lem:LadjU}
    If $\D(\oplus \times \id)$ is always $T$-complete, the functor $\L \maps \C \to \Sys_\F$ factors through $\TComp$, the full subcategory of $\Sys_\F$ spanned by the $T$-complete systems, and $\L$ is left adjoint to the forgetful functor $U \maps \TComp \to \C$.
\end{lem}
\begin{proof}
    Since $\D(\oplus \times \id_B) = \L_B$ is always $T$-complete, $\L$ can be factored through $\TComp$. 
    Each map $0_\oplus \times \id_B \maps B \to T \times B = (U \circ \L)(B)$ is a universal arrow from $B$ to $U$, $(\L_B, 0_\oplus \times \id _B \maps B \to U(\L_B))$ since if you had another pair $(f \maps A \to \F A, \gamma \maps B \to Uf = A)$, then the $T$-completeness of $f$ gives a unique map of systems $\overline \gamma \maps \L_B \to f$ such that $U(\overline \gamma) \circ (0_\oplus \times \id) = \gamma$.
    The maps $0_\oplus \times \id_B \maps B \to T \times B = U(\L_B)$ clearly form a natural transformation $0_\oplus \times \id \maps id_\C \To U \circ \L$. 
    The result then follows from \cite[Ch.\ IV, Theorem 2]{MacLane}.
\end{proof}

\begin{example}\label{ex:classicaladjunction}
    Recall the setting for dynamic stability for vector fields on manifolds given in \cref{ex:classicstability2}. In this setting, the free coalgebra $\L_X$ has a derivative of 1 on the time axis and 0 on the manifold. by the Picard--Lindel\"of--Cauchy--Lipschitz criterion of differential equations, we know it must have unique solutions. Since $\L_X$ always has solutions by functoriality of $\L$, then we conclude that $\L_X$is always $T$-complete, therefore the adjunction of \cref{lem:LadjU} is valid in this setting.
\end{example}

\begin{example}
\label{ex:PLCL}
    Continuing \cref{ex:LTS3}, the derivative of the flow is the graph formed by taking the output at each point $x\in X$ after a flow of $a\in A$ and setting a unique arrow $x\xrightarrow{a}\flow(x,a)$. For any flow $\flow:T\times X\to X$, it is a routine exercise to show $\D\flow$ such system is deterministic and complete (in the sense that every node admit an arrow with every label). Those two attribute are equivalent to being $T$-complete. This show that we have the stronger condition of uniqueness in  \cref{thm:existenceuniqueness}.
\end{example}

\begin{lem}
\label{lem:integralfunctor}
    Let $ f \maps A \to \F A$ and $ g \maps B \to \F B$ be $T$-complete $\F$-systems. If $p \maps A \to B$ is a map of $\F$-systems $ f \to  g$, then it is also a map of $T$-flows $\int  f \to \int  g$. In this way, $\I$ extends to a functor $\TComp \to \TFlow$. 
\end{lem}
\begin{proof}
    That $\I f$ and $\I g$ are $T$-flows follows from \cref{lem:Tselfcomplete}. 
    
    Consider the following diagram.
    \[\begin{tikzcd}
        \F(T \times A)
        \arrow[r, "\F(\id \times p)"]
        &
        \F(T \times B)
        \arrow[r, "\F(\I g)"]
        & 
        \F B
        \\
        T \times A 
        \arrow[u, "\L_A"]
        \arrow[r, "\id \times p"']
        & 
        T \times B
        \arrow[u, "\L_B"]
        \arrow[r, "\I g"']
        &
        B
        \arrow[u, "g"']
    \end{tikzcd}\]
    The left diagram commutes due to functoriality of $\L$. 
    The right diagram commutes by the definition of $\I$.
    
    Consider the following diagram.
    \[\begin{tikzcd}
        \F(T \times A)
        \arrow[r, "\F(\I f)"]
        &
        \F A
        \arrow[r, "\F p"]
        & 
        \F B
        \\
        T \times A 
        \arrow[u, "\L_A"]
        \arrow[r, "\I f"']
        & 
        A
        \arrow[u, "f"]
        \arrow[r, "p"']
        &
        B
        \arrow[u, "g"']
    \end{tikzcd}\]
    The left diagram commutes by definition of $\I$. 
    The right diagram commutes by the assumption that $p$ is a map of $\F$-systems.

    The composites $\I g \circ (\id \times p)$ and $p \circ \I f$ are therefore both solutions of $g$. It is clear they both have initial condition $p$.  
    \[\begin{tikzcd}[column sep = small]
    	A
        \arrow[rr, "0_\oplus \times \id"]
        \arrow[dddrr, "p"', bend right = 15]
        && 
        T \times A
        \arrow[rrr, "\L_A"]
        \arrow[dl, "\I f"description]
        &&& 
        \F(T \times A)
        \arrow[dl, "\F\I f"]
        \arrow[ddr, "\F(\id \times p)"]
        \\& 
        A
        \arrow[ddr, "p"]
        \arrow[rrr, "f", pos = 0.8]
        &&& 
        \F A
        \arrow[ddr, "\F p", pos = 0.3]
        \\&&& 
        T \times B
        \arrow[dl, "\I g"']
        &&& 
        \F(T \times B) 
        \arrow[dl, "\F\I g"]
        \\&& 
        B
        \arrow[rrr, "g"']
        &&& 
        \F B
        \arrow[crossing over, from = 1-3, to = 3-4, "\id \times p", pos = 0.3]
        \arrow[crossing over, from = 3-4, to = 3-7, "\L_B", pos = 0.7]
    \end{tikzcd}\]
    By $T$-completeness of $g$, these composites must be equal, and thus $p$ is a map of $T$-flows.
\end{proof}

\begin{rmk}
\label{rmk:integral}
    As noted in \cref{rmk:Tx-monad}, the monad induced by the adjunction $U\L$ of \cref{lem:LadjU} is $T \times -$, and the Eilenberg-Moore category is $\C^{T \times -} \cong \TFlow$. The canonical comparison functor $\TComp \to \TFlow$ is precisely the functor $\I$ given in \cref{lem:integralfunctor}.
\end{rmk}

\begin{lem}
\label{lem:ID=1}
    Assume a setting with compatible unit clock. Let $\flow \maps T \times E \to E$ be a $T$-flow. If $\D\flow$ is $T$-complete, then $\I\D\flow = \flow$.
\end{lem}
\begin{proof}
    By \cref{lem:flowissolution}, $\flow$ is a system map $\L_E \to \D\flow$, and it has $\flow \circ (0_\oplus \times \id_E) = \id_E$ by definition of $T$-flows. By \cref{def:integral}, $\I\D\flow$ is also a system map $\L_E \to \D\flow$ that has $\flow \circ (0_\oplus \times \id_E) = \id_E$. Since $ f$ is $T$-complete, these must be equal.
\end{proof}

\begin{proof}
[Proof of \cref{thm:existenceuniqueness}]
    We have shown in \cref{lem:Tselfcomplete}, that $\I f_E$ is a solution to $ f_E$. We showed that $\I$ is a functor $\TComp \to \TFlow$ in  \cref{def:integral} and \cref{rmk:integral}.

    That the derivative defines a functor $\D \maps \TFlow \to \Sys_\F$ can be checked directly, or through the universal property of $\Sys_\F$ given in \cref{rmk:functtocoalgisnat}. The assumption that $\D\flow$ is $T$-complete implies that this corestricts to the form $\D \maps \TFlow \to \TComp$ since $\D(\oplus \times \id_X) = \L_X$ by \cref{lem:associatedcoalg}.

    By \cref{lem:TcompEquivalent}, $\D\flow$ has the uniqueness property, and thus $\I f_E$ is the only solution of $ f_E$.
    We showed in \cref{lem:ID=1} that $\I\D \flow = \flow$. In \cref{eq:integral}, the upper path is $\D\I f_E$ by definition, and the lower is $ f_E$. Thus $\I\D  f_E =  f_E$.

    Since $\D\I = \id_{\TComp}$ and $\I\D_{\TFlow}$, $\D$ and $\I$ form an isomorphism of categories.
\end{proof}

\begin{proof}[Proof of \cref{cor:monadic}]
    In \cref{rmk:integral}, we noted that $\I \maps \TComp \to \TFlow$ is the canonical comparison functor given by the universal property of $\TFlow$ as the Eilenberg-Moore category of the monad $U\circ\L = T \times -$. \cref{thm:existenceuniqueness} shows that it is an equivalence.  
\end{proof}

\section{Converse Lyapunov Theorem}
\label{sec:converse}

Finally, we address the converse of Lyapunov's theorem. Lyapunov's theorem tells us that a Lyapunov function certifies the stability of the equilibrium point of interest, and the converse assures us that there are no stable points that are missed if we just consider Lyapunov functions.

In \cite{CLT1}, we proved a converse for flow settings under the additional assumption that $R$ admits certain suprema. Specifically, we showed that for a $T$-flow $\flow \maps T \times E \to E$ in a setting for stability, if a point $x^*\in E$ is stable, then there is a morphism which is positive definite relative to $x^*$ and decrescent relative to the flow. To complete the coalgebraic side of the story, we presently demonstrate that if a positive definite morphism is decrescent relative to the integral flow of a $T$-complete $\F$-system $f \maps E \to \F E$, then it is decrescent relative to the system $f$ itself. 

As in the flow story, obtaining a converse requires some additional assumptions. 

\begin{defn}
\label{def:ConverseSetting}
    A \define{converse setting for stability} is a monoidal setting for dynamic stability satisfying the follow additional axioms:
    \begin{itemize}
        \item[\textbf{D7:}] (lax monoidal functor) The stationary system associated to the terminal object $0_1 \maps 1 \to \F1$ is neutral with respect to the laxator in the sense that the following diagrams commute.
        \[
        \begin{tikzcd}
            1 \times \F X
            \arrow[r, "0_1 \times \id"]
            \arrow[d, "\lambda_{\F X}"']
            &
            \F 1 \times \F X
            \arrow[d, "\psi_{1,X}"]
            \\
            \F X
            &
            \F(1 \times X)
            \arrow[l, "\F\lambda_X"]
        \end{tikzcd}
        \qquad
        \begin{tikzcd}
            \F X \times 1
            \arrow[r, "\id \times 0_1"]
            \arrow[d, "\rho_{\F X}"']
            &
            \F X \times \F 1
            \arrow[d, "\psi_{X,1}"]
            \\
            \F X
            &
            \F(X \times 1)
            \arrow[l, "\F\rho_X"]
        \end{tikzcd}
        \]
        \item[\textbf{D8:}] (unitary) $0_1 \maps 1 \to \F1$ is $T$-complete.
        \item[\textbf{D9:}] (order-preserving) For $A \in \C$, if $f \leq g \maps A \to R$, then $\F f \leq \F g \maps \F A \to \F R$. 
    \end{itemize}
\end{defn}

\begin{rmk}
    Along with Axioms \textbf{D5} and \textbf{D6}, the first axiom says both that $\F$ is lax monoidal and that $0$ is a \define{monoidal natural transformation}, ${0} \maps \id_\C \To \F$. As noted in \cref{rmk:tensorproduct}, the laxator $\psi$ defines an associative tensor product $\otimes_\psi$ on $\Sys_\F$. Axiom \textbf{D7} implies that $0_1 \maps 1 \to \F 1$ is a unit object with respect to this tensor product, and thus $\Sys_\F$ carries a monoidal structure. This also implies that $\L_1 =  1_T \otimes_\psi 0_1 \cong  1_T$. 

    Recall from \cref{rmk:functtocoalgisnat} that the data of the natural transformation $ 0 \maps \id_\C \to \F$ can be rearranged into a functor of the form $ 0 \maps \C \to \Sys_\F$ which sends an object $A$ to its corresponding stationary system $ 0_A$. It is straightforward to check that monoidalness of a natural transformation $\id_\C \To \F$ corresponds to strict monoidalness of the induced functor $\C \to \Sys(\F)$.
\end{rmk}

\begin{lem}
\label{lem:Dpi=0}
    In a converse setting, $0_X=\D\pi_X$ for all $X \in \C$. 
\end{lem}
\begin{proof}
    Since $0_1$ is $T$-complete, then it has a solution flow $T \times 1 \to 1$, which has to be the unique terminal map.
    \[
    \begin{tikzcd}
        1
        \arrow[r, "0_\oplus"]
        \arrow[dr, "\id"']
        &
        T
        \arrow[r, " 1_T"]
        \arrow[d, "!"]
        &
        \F T
        \arrow[d, "\F!"]
        \\&
        1
        \arrow[r, "0_1"']
        &
        \F1
    \end{tikzcd}
    \]
    We then construct the following commuting diagram.
    \[
    \begin{tikzcd}
        X
        \arrow[r, "0_\oplus \times \id"]
        \arrow[dr, "\id"']
        &
        T \times X
        \arrow[rr, "\L_X", bend left]
        \arrow[r, "1_T \times 0_X"']
        \arrow[d, "\pi"]
        &
        \F T \times \F X
        \arrow[r, "\psi_{T,X}"]
        \arrow[d, "\F! \times \id"]
        &
        \F(T \times X)
        \arrow[d, "\F\pi_X"]
        \\&
        X
        \arrow[r, "0_1 \times 0_X"']
        \arrow[rr, bend right, "0_{X}"']
        &
        \F1 \times \F X
        \arrow[r, "\psi_{1,X}"']
        &
        \F X
    \end{tikzcd}
    \]
    The upper path is exactly the definition of $\D\pi_X$, so the whole diagram says $\D\pi_X = 0_X$.
\end{proof}

\begin{lem}
    Assume a converse setting. Let $ f \maps E \to \F E$ be an $\F$-system and $\flow \maps T \times E \to E$ be a solution flow of $ f$. If $x^* \maps 1 \to E$ is an equilibrium point for $\flow$, then it is an equilibrium point for $ f$ in the sense that the following diagram commutes.
    \begin{equation}
    \label{eq:coalgeqpt}
    \begin{tikzcd}
        \F1
        \arrow[r, "\F x^*"]
        &
        \F E
        \\
        1
        \arrow[u, "0_1"]
        \arrow[r, "x^*"']
        &
        E
        \arrow[u, " f"']
    \end{tikzcd}
    \end{equation}
\end{lem}
\begin{proof}
    In the diagram below, the top and bottom cells commute by the definition of $\D$. The leftmost square commutes by product. The center square commutes by naturality of $\L$. The rightmost square commutes by applying $\F$ to the condition that $x^*$ is an equilibrium point of $\flow$. 
    \[\begin{tikzcd}
        1
        \arrow[bend left, rrr, "\D!=0_1"]
        \arrow[r, "0_\oplus \times \id"']
        \arrow[d, "x^*"']
        &
        T \times 1
        \arrow[r, "\L_1"]
        \arrow[d, "\id \times x^*"]
        &
        \F(T \times 1)
        \arrow[r, "\F!"]
        \arrow[d, "\F(\id \times x^*)"]
        &
        \F1
        \arrow[d, "\F x^*"]
        \\
        E
        \arrow[rrr, bend right, "\D\flow =  f"']
        \arrow[r, "0_\oplus \times \id"'']
        &
        T \times E
        \arrow[r, "\L_E"']
        &
        \F(T \times E)
        \arrow[r, "\F\flow"']
        &
        \F E
    \end{tikzcd}\]
    By \cref{lem:Dpi=0}, $\D!=0_1$, and $\D\flow = f$ since $\flow$ is a solution of $ f$. 
\end{proof}

The lemma above has a converse when $ f$ is $T$-complete, in which case we can replace $\flow$ with $\I f$. 

\begin{lem}
    Assume a converse setting. Let $ f \maps E \to \F E$ be a $T$-complete $\F$-system. If $x^* \maps 1 \to E$ is an equilibrium point for $ f$ in the sense of \cref{eq:coalgeqpt}, then it is an equilibrium point for $\I f$.
\end{lem}
\begin{proof}
    Consider the following diagram
    \[
    \begin{tikzcd}
        \F(T \times 1)
        \arrow[r, "\F!"]
        &
        \F1
        \arrow[r, "\F x^*"]
        &
        \F E
        \\
        T \times 1
        \arrow[u, "\L_1"]
        \arrow[r, "!"']
        &
        1
        \arrow[u,"0_1"]
        \arrow[r, "x^*"']
        &
        E
        \arrow[u, " f"']
    \end{tikzcd}
    \]
    The left square commutes since $\pi$ is a solution of $0_1$ by \cref{lem:Dpi=0}. The right square commutes by the assumption that $x^*$ is an equilibrium point of $ f$. Thus the composite \[T \times 1 \xrightarrow!1 \xrightarrow{x^*} E\] is a map of $\F$-system $\L_1 \to  f$. It is clear that $x^* \circ ! \circ (0_\oplus \times id) = x^*$. 
    
    Consider the following diagram.
    \[
    \begin{tikzcd}
        \F(T \times 1)
        \arrow[r, "\F(\id \times x^*)"]
        &
        \F(T \times E)
        \arrow[r, "\F \I f"]
        &
        \F E
        \\
        T \times 1
        \arrow[u, "\L_1"]
        \arrow[r, "\id \times x^*"']
        &
        T \times E
        \arrow[u,"\L_E"]
        \arrow[r, "\I f"']
        &
        E
        \arrow[u, " f"']
    \end{tikzcd}
    \]
    The left square commutes by functoriality of $\L$, or equivalently the naturality of the transformation corresponding to $\L$. The right square commutes since $\I f$ is a solution flow of $ f$. Thus the composite 
    \[T \times 1 \xrightarrow{\id \times x^*} T \times E \xrightarrow{\I f} E\]
    is a map of $\F$-system $\L_1 \to  f$. Since $\I f$ is a $T$-flow, we have $x^* \circ \I f \circ (\id \times x^*) = x^*$. 

    The composites $\I f \circ (\id \times x^*)$ and $x^* \circ !$ are therefore solutions of $ f$ with the same initial condition.
    \[\begin{tikzcd}[column sep = small]
    	1
        \arrow[rr, "0_\oplus \times \id"]
        \arrow[dddrr, "x^*"', bend right = 15]
        && 
        T \times 1
        \arrow[rrr, "\L_E"]
        \arrow[dl, "!"description]
        &&& 
        \F(T \times 1)
        \arrow[dl, "\F!"]
        \arrow[ddr, "\F(\id \times x^*)"]
        \\& 
        1
        \arrow[ddr, "x^*"]
        \arrow[rrr, "0_1", pos = 0.8]
        &&& 
        \F 1
        \arrow[ddr, "\F x^*", pos = 0.3]
        \\&&& 
        T \times E
        \arrow[dl, "\I f"']
        &&& 
        \F(T \times E) 
        \arrow[dl, "\F\I f"]
        \\&& 
        E
        \arrow[rrr, " f"']
        &&& 
        \F E
        \arrow[crossing over, from = 1-3, to = 3-4, "\id \times x^*", pos = 0.3]
        \arrow[crossing over, from = 3-4, to = 3-7, "\L_E", pos = 0.7]
    \end{tikzcd}\]
    Thus by $T$-completeness of $ f$, $x^*$ is an equilibrium point of $\I f$ as desired.
\end{proof}

\begin{lem}
\label{lem:decrescent}
    Assume a converse setting. Let $ f \maps E \to \F E$ be a $T$-complete system. If $V \maps E \to R$ is flow decrescent for $\I f$ in the sense of \cref{def:lyapunov}, then $V$ is system decrescent for $ f$ in that it makes \cref{eq:SysDecrescent} lax commute.
\end{lem}
\begin{proof}
    In the diagram
    \[
    \begin{tikzcd}
        E
        \arrow[r, "0_\oplus \times \id_E"]
        \arrow[d, "V"']
        &
        T \times E 
        \arrow[r, "\L_E"]
        \arrow[d, "\id \times V"']
        & 
        \F(T \times E)
        \arrow[r, "\F\flow"]
        \arrow[d, "\F(\id \times V)"']
        &
        \F E
        \arrow[d, "\F V"]
        \\
        R
        \arrow[r, "0_\oplus \times \id_R"']
        &
        T \times R 
        \arrow[r, "\L_R"']
        & 
        \F(T \times R)
        \arrow[r, "\F\pi_R"']
        \ar[ur,Rightarrow]
        &
        \F R
    \end{tikzcd}
    \]
    the left square commutes trivially, the middle square commutes by functoriality of $\L$, and the right square lax commutes by assumption and Axiom \textbf{D9}. The top row is $\D\flow$ by definition. The bottom row is the system $0_R$ by \cref{lem:Dpi=0}.
\end{proof}

\begin{thm}[Converse Lyapunov Theorem]
\label{thm:Lyapconvsys}
    Assume a converse setting such that $R$ has local suprema commuting with whiskering. Let $x^* \maps 1 \to E$ be an equilibrium point of a $T$-complete system $ f \maps E \to \F E$. If $x^*$ is stable, then there exists a positive definite morphism $V \maps E \to R$ such that \cref{eq:SysDecrescent} lax commutes.
\end{thm}
\begin{proof}
    Assume that $x^*$ is a stable equilibrium point of $\I f$. Then by the converse Lyapunov theorem for flows given in \cite{CLT1}, $V \coloneq \sup_T \xnorm{\flow}$ is a positive definition morphism $E \to R$ which is decrescent is the flow sense of \cref{def:lyapunov}. Then by \cref{lem:decrescent}, since $V$ is flow decrescent for $\I f$, then it is system decrescent in the sense of \cref{eq:SysDecrescent} for $ f$. 
\end{proof}

\subsection*{Acknowledgments}

This research was supported by the Air Force Office of Scientific Research under the Multidisciplinary University Research Initiative grant Hybrid Dynamics ‐ Deconstruction and Aggregation (HyDDRA). This research was supported by a grant from Wallonie-Bruxelles International (WBI). We would like to thank David Spivak and Paulo Tabuada for the numerous discussions on the categorical framing of Lyapunov theory. We would also like to thank \'Alvaro Rodr\'iguez Abella, John Baez, and Todd Trimble. 

\bibliographystyle{plain}
\bibliography{references}

\end{document}